\documentclass[10pt]{article}
\usepackage{amsmath}
\usepackage{amsthm}
\usepackage{amssymb}
\newcommand{\Bm}{\ensuremath{\beta_-(\gamma)}}
\newcommand{\Bp}{\ensuremath{\beta_+(\gamma)}}
\newcommand{\R}{\ensuremath{\mathbb{R}}}
\newcommand{\N}{\ensuremath{\mathbb{N}}}
\newcommand{\Ne}{\ensuremath{\mathcal{N}}}
\newcommand{\M}{\ensuremath{\mathcal{M}}}
\newcommand{\X}{\ensuremath{X_0^\frac{\alpha}{2}(\Omega)}}
\newcommand{\h}{\ensuremath{H^\frac{\alpha}{2}(\R^n)}}
\newcommand{\I}{\ensuremath{I_\lambda}}
\newcommand{\fy}{\ensuremath{\phi_{\lambda,p}}}

\begin{document}

\theoremstyle{definition}
\newtheorem{definition}{Definition}[section]
\theoremstyle{definition}
\newtheorem{theorem}[definition]{Theorem}
\newtheorem{lemma}[definition]{Lemma}
\newtheorem{proposition}[definition]{Proposition}
\newtheorem{corollary}[definition]{Corollary}
\newtheorem{remark}[definition]{Remark}
\theoremstyle{definition}
\newtheorem{example}[definition]{Example}
\newtheorem{step}{Step}
\newtheorem{claim}[definition]{Claim}
\newtheorem{case}{Case}

\title{Multiple Positive Solutions for Nonlocal Elliptic Problems Involving the Hardy Potential and Concave-Convex Nonlinearities}
\author{Shaya Shakerian  \\
{\it\small Department of Mathematics}\\
{\it\small  University of British Columbia}\\
{\it\small Vancouver BC Canada V6T 1Z2}\\
{\it\small shaya@math.ubc.ca}\vspace{1mm}
\\\\
}
{}
 \maketitle
\begin{center}
{\bf\small Abstract}

\vspace{3mm}
\hspace{.05in}\parbox{4.5in}
{{\small
In this paper, we study the existence and multiplicity of solutions for the following fractional problem involving the Hardy potential and concave-convex nonlinearities:

\begin{equation*}\label{Main problem} 
\left\{\begin{array}{rl}
({-}{ \Delta})^{\frac{\alpha}{2}}u- \gamma \frac{u}{|x|^{\alpha}}= \lambda f(x) |u|^{q - 2} u + g(x) {\frac{|u|^{p-2}u}{|x|^s}} & \text{in }  {\Omega}\\
 u=0 \,\,\,\,\,\,\,\,\,\,\,\,\,\,\,\,\,\,\,\,\,\,\,\,\,\,\,\,\,\,\,\,\,\,\,\,\,\,\,\,\,\,\,\,\,\,\,\,\,\,\,\,\,\,\,\,\,\,\,\,\,\,\,\,\,\,\,\,  & \text{in }  \R^n \setminus \Omega,
\end{array}\right.
\end{equation*}

where $\Omega \subset \R^n$ is a smooth bounded domain in $\R^n$ containing $0$ in its interior, and $f,g \in C(\overline{\Omega})$ with $f^+,g^+ \not\equiv 0$ which may change sign in $\overline{\Omega}.$ We use the variational methods and  the Nehari manifold decomposition to prove that this problem has at least two positive solutions for $\lambda$ sufficiently small. The variational approach requires that $0 < \alpha <2,$   $ 0 <s < \alpha <n,$ $ 1<q<2<p \le 2_{\alpha}^*(s):= \frac{2(n-s)}{n-\alpha},$ and $ \gamma < \gamma_H(\alpha) ,$ the latter being the best fractional Hardy constant on $\R^n.$ 
 }}

\end{center}

\section{Introduction}

In this paper, we investigate the multiplicity results of positive solutions for the following fractional elliptic problem involving the Hardy potential and concave-convex non-linearities:

\begin{equation}\label{Main problem}
\left\{\begin{array}{rl}
({-}{ \Delta})^{\frac{\alpha}{2}}u- \gamma \frac{u}{|x|^{\alpha}}= \lambda f(x) |u|^{q - 2} u + g(x) {\frac{|u|^{p-2}u}{|x|^s}} & \text{in }  {\Omega}\\
 u=0 \,\,\,\,\,\,\,\,\,\,\,\,\,\,\,\,\,\,\,\,\,\,\,\,\,\,\,\,\,\,\,\,\,\,\,\,\,\,\,\,\,\,\,\,\,\,\,\,\,\,\,\,\,\,\,\,\,\,\,\,\,\,\,\,\,\,\,  & \text{in }  \R^n \setminus \Omega,
\end{array}\right.
\end{equation}

where $\Omega \subset \R^n$ is a smooth bounded domain in $\R^n$ containing $0$ in its interior, $0 < \alpha <2,$   $ 0 <s < \alpha <n,$ $ 1<q<2<p \le 2_{\alpha}^*(s):= \frac{2(n-s)}{n-\alpha},$ $ \gamma < \gamma_H(\alpha) = 2^\alpha \frac{\Gamma^2(\frac{n+\alpha}{4})}{\Gamma^2(\frac{n-\alpha}{4})} ,$ the later being the best fractional Hardy constant on $\R^n,$
 and $f,g \in C(\overline{\Omega})$ with $f^+,g^+ \not\equiv 0$ (they are possibly change sign in $\overline{\Omega}$). 

\newpage

Addressing the questions regarding the effect of concave-convex non-linearities on the number of positive solutions for non-local elliptic problems has been the subject of several studies; see [1-4]
 and \cite{Quaas-Xia}. Barrios-Medina-Peral \cite{Barrios-Medina-Peral} studied the sub-critical case of (\ref{Main problem}), 
 and proved that there exists  $\Lambda>0$ such that the problem has at least two solutions for all $0<\lambda<\Lambda$ when  $f(x)=g(x) \equiv 1,$ $s=0,$ and  $\gamma<\gamma_H(\alpha).$ The same results have been obtained by Barrios et al. in  \cite{B-C-D-S 1}  in the critical case, but in the absence of the Hardy  and singularity terms, i.e., when  $\gamma = s = 0.$ Recently, Zhang-Liu-Jiao \cite{Zhang-Liu-Jiao} extended the results of \cite{B-C-D-S 1} to the problem involving the sign-changing weight function $f(x) \in C(\overline{\Omega})$ with $f^+ \not \equiv 0,$ and  $g(x) \equiv 1.$  

When $\alpha=2,$ i.e., in the case of the standard Laplacian, problem (\ref{Main problem}) has been studied extensively in the last decade; see for example \cite{Wang-Wei-Kang}, \cite{Wu}, \cite{Wu-Critical}, \cite{Tarantello} and references therein.

 We point out that the non-local problems are still much less understood than their local counterpart. The aim of this paper is to consider the remaining cases and generalize  the  results of \cite{Barrios-Medina-Peral} and \cite{Zhang-Liu-Jiao} to the problem involving the Hardy potential, Hardy-Sobolev singularity term, and also sign-changing functions $f(x)$ and $g(x)$. More pricesly, we study  problem (\ref{Main problem}) in the sub-critical (i.e., when $ 2 < p < 2_\alpha^*(s)$) and critical (i.e., when $ p  = 2_\alpha^*(s)$) case, separately. Using the decomposition of the Nehari manifold as $\lambda$ varies, introduced by Tarantello \cite{Tarantello}, we will prove that the problem  has at least two positive solutions for $\lambda$ sufficiently small.

We first prove the following theorem:


\begin{theorem}\label{Theorem - Sub-Critical}
Let $0< \alpha <2$ and $0\le s<\alpha <n.$  Suppose $ 1<q<2<p < 2_{\alpha}^*(s)$ and $\gamma< \gamma_H(\alpha).$  Then, there exits  $\Lambda > 0$ such that  problem (\ref{Main problem}) has at least two positive solutions for any $\lambda \in (0,\Lambda).$
\end{theorem}

The critical case 
 is more challenging and requires information about   the asymptotic behaviour of solutions of the following limiting problem at zero and infinity:

\begin{equation}\label{Limiting problem: introduction}
\left\{\begin{array}{rl}
({-}{ \Delta})^{\frac{\alpha}{2}}u- \gamma \frac{u}{|x|^{\alpha}}= {\frac{u^{2_{\alpha}^*(s)-1}}{|x|^s}} & \text{in }  {\R^n}\\
 u \ge 0 \,\,\,\,\,\,\,\,\,\,\,\,\,\,\,\,\,  & \text{in }  \mathbb{R}^n,
\end{array}\right.
\end{equation}

where $0<\alpha<2$, $ 0 \le s < \alpha$,  $ {2_{\alpha}^*(s)}={\frac{2(n-s)}{n-{\alpha}}}$,   
$ 0 \le \gamma < \gamma_H(\alpha)=2^\alpha \frac{\Gamma^2(\frac{n+\alpha}{4})}{\Gamma^2(\frac{n-\alpha}{4})}.$ We get around the difficulty by working with certain asymptotic estimates for solutions of (\ref{Limiting problem: introduction}) recently
obtained by the author and et al. in  \cite{Ghoussoub-Robert-Zhao-Shaya}; see Theorem \ref{Theorem-Upper and lower bound for the solution on Rn}. In order to use the results of \cite{Ghoussoub-Robert-Zhao-Shaya}, we may assume 
$g(x) \equiv 1.$  Problem (\ref{Main problem}) therefore can be written as follows:
 
\begin{equation}\label{Main problem-critical}
\left\{\begin{array}{rl}
({-}{ \Delta})^{\frac{\alpha}{2}}u -\gamma \frac{u}{|x|^\alpha}= \lambda f(x) |u|^{q - 2} u + {\frac{|u|^{2_{\alpha}^*(s)-2}u}{|x|^s}} & \text{in }  {\Omega}\\
 u=0 \,\,\,\,\,\,\,\,\,\,\,\,\,\,\,\,\,\,\,\,\,\,\,\,\,\,\,\,\,\,\,\,\,\,\,\,\,\,\,\,\,\,\,\,\,\,\,\,\,\,\,\,\,\,\,\,\,\,\,\,\,\,\,\,  & \text{in }  \R^n \setminus \Omega.
\end{array}\right.
\end{equation}

We then establish the following:

\begin{theorem}\label{Theorem - Critical}
 Let $0< \alpha <2$ and $0\le s<\alpha <n.$  Suppose $ 1<q<2<p = 2_{\alpha}^*(s)$ and $0 \le \gamma< \gamma_H(\alpha).$ Then, there exits  $\Lambda^* > 0 $ such that  problem (\ref{Main problem-critical}) has at least two positive solutions for any $\lambda \in (0,\Lambda^*).$
\end{theorem}

\section{Functional Setting }

We start by recalling and  introducing suitable function spaces for the variational principles that will be needed in the sequel.  We first recall that the non-local operator $(-\Delta)^{\frac{\alpha}{2}}$ is  defined as

\begin{equation*}\label{Definition of the fractional laplacian }
(-\Delta)^{\frac{\alpha}{2}} u (x) = c(n,\alpha) P.V. \int_{\R^n} \frac{u(x)- u(y)}{|x-y|^{n+\alpha}} dy  \quad \text{ for } x \in \R^n, \quad   \text{ where }  c(n,\alpha)=\frac{2^{\alpha-1}\Gamma\left(\frac{n+\alpha}{2}\right)}{\pi^{\frac{n}{2}}\left|\Gamma\left(-\frac{\alpha}{2}\right)\right| }.
\end{equation*}

 We denote by $H^\frac{\alpha}{2}(\R^n)$ the classical fractional Sobolev space endowed with the so-called Gagliardo norm

$$\|u\|_{\h} = \|u\|_{L^2(\R^n)}+ \left( \int_{\R^n}  \int_{\R^n} \frac{|u(x)- u(y)|^2}{|x-y|^{n+\alpha}} dx dy\right)^{\frac{1}{2}}.$$

For $\alpha \in (0,2)$, the fractional Sobolev space $H_0^{\frac{\alpha}{2}}(\R^n)$ is defined as the completion of $C_c^{\infty}(\R^n)$ under the norm 
$$\|u\|_{H_0^{\frac{\alpha}{2}}(\R^n)}^2= \int_{\mathbb{R}^n}|2\pi \xi |^{\alpha} |\mathcal{F}u(\xi)|^2 d\xi =\int_{\mathbb{R}^n} |(-\Delta)^{\frac{\alpha}{4}}u|^2 dx.$$

Let now $\Omega \subset \R^n$ be a smooth bounded domain. We  define the space $\X$ as

$$\X = \left\{ u \in \h : u= 0 \text{ a.e. in } \R^n \setminus \Omega   \right\},$$

and consider the following norm in $\X:$

$$\|u\|_{\X} = \left( \int_{\R^n}  \int_{\R^n} \frac{|u(x)- u(y)|^2}{|x-y|^{n+\alpha}}dx dy \right)^{\frac{1}{2}}.$$

We  recall that $(\X, \| \ . \ \|_{\X} )$ is a Hilbert space with the scalar product 

$$\langle u,v \rangle_{\X} =  \int_{\R^n}  \int_{\R^n} \frac{(u(x)- u(y))(v(x) - v(y))}{|x-y|^{n+\alpha}}dx dy.$$

\begin{remark}
It was shown in \cite{Bisci-Radulescu-Servadei} that the sub-space $C^\infty_0(\Omega) $
is  dense in $\X.$ So, we can consider $\X$ as the completion of $C^\infty_0(\Omega) $ with with respect to the norm $\| \ . \ \|_{\X}.$
\end{remark}



\begin{definition}
We say $u \in \X$ is  a weak solution of (\ref{Main problem}), if for every $\phi \in \X,$ we have

$$c(n,\alpha) \langle u, \phi \rangle_{\X} -  \gamma\int_{\Omega}\frac{u \phi }{|x|^{\alpha}} dx = \lambda \int_{\Omega}  f(x) |u|^{q-1} \phi dx- \int_{\Omega} g(x) \frac{|u|^{p-1} \phi }{|x|^{s}}  dx.$$ 
\end{definition}

The energy functional corresponding to  (\ref{Main problem}) is  
\begin{equation}\label{Fuunctional I_Lambda}
\begin{aligned}
I_{\lambda,p}(u)&= \frac{c(n,\alpha)}{2} \| u\|^2_{X_0^{\alpha}(\Omega) } - \frac{\gamma}{2}\int_{\Omega}\frac{|u|^{2}}{|x|^{\alpha}} dx -\frac{\lambda}{q} \int_{\Omega}  f(x) |u|^q dx-\frac{1}{p}\int_{\Omega} g(x) \frac{|u|^{p}}{|x|^{s}}  dx \\ 
\end{aligned}.
\end{equation}

Recall that any critical point $u$ of $I_{\lambda,p}(u)$ is a weak solution for (\ref{Main problem}). The starting point of the study of existence of weak solutions to problem (\ref{Main problem}) is therefore the following fractional  inequalities which will guarantee that the above functional is well defined and bounded below on the right function spaces.

We start with the fractional Sobolev inequality \cite{Cotsiolis-Tavoularis}, which asserts that for $n > \alpha$ and $ 0<\alpha<2$, there exists a constant $S(n,\alpha) >0 $ such that

\begin{equation*}
\hbox{$\|u\|_{{L^{2_{\alpha}^*}}(\mathbb{R}^n)}^{^2} \leq S(n,\alpha)\int_{\R^n}  \int_{\R^n} \frac{|u(x)- u(y)|^2}{|x-y|^{n+\alpha}} dxdy$ \qquad for $u \in H^{\frac{\alpha}{2}} (\mathbb{R}^n),$}
\end{equation*}
where $2_{\alpha}^* = \frac{2n}{n-\alpha}$. Another important inequality is the fractional  Hardy inequality \cite{Herbst}, which states that under the same conditions on $n$ and $\alpha$, there exists a constant $ \gamma (n,\alpha)>0$ such that

\begin{equation}
\gamma_H (\alpha) \int_{\R^n}\frac{|u|^{2}}{|x|^{\alpha}} dx \le \int_{\mathbb{R}^n}|\xi |^{\alpha} |\mathcal{F}u(\xi)|^2 d\xi  \qquad   \text{ for }u \in C^\infty_0(\R^n),
\end{equation}

where $ \mathcal{F}u(x) = \frac{1}{(2 \pi)^{\frac{n}{2}}} \int_{\R^n} e^{- i \xi x } u(x) dx$ is the Fourier transform of $u.$  It was shown in \cite{Herbst} that $ \gamma_H(\alpha) := 2^\alpha \frac{\Gamma^2(\frac{n+\alpha}{4})}{\Gamma^2(\frac{n-\alpha}{4})}$ is the best constant in the above fractional Hardy inequality. Note that $\gamma_H(\alpha)$ converges to the best 
classical Hardy constant $\frac{(n-2)^2}{4}$  when $\alpha \to 2.$

 By Proposition 3.6 in \cite{Hitchhikers guide}, for any $u \in \h ,$ we have  the following relation between the fractional Laplacian operator $\displaystyle (-\Delta)^{\frac{\alpha}{2}}$ and the fractional Sobolev space $\h$ :
\begin{equation}
 \int_{\mathbb{R}^n}| \xi |^{\alpha} |\mathcal{F}u(\xi)|^2 d\xi= c(n,\alpha)\int_{\R^n}  \int_{\R^n} \frac{|u(x)- u(y)|^2}{|x-y|^{n+\alpha}} dxdy.
\end{equation}

The fractional Hard inequality then can be written as

\begin{equation}\label{Fractional Hardy Inequality}
\gamma_H(\alpha) \int_{\R^n}\frac{|u|^{2}}{|x|^{\alpha}} dx  \le c(n,\alpha) \int_{\R^n}  \int_{\R^n} \frac{|u(x)- u(y)|^2}{|x-y|^{n+\alpha}} dxdy  \qquad \text{ for }  u \in C^\infty_0(\R^n).
\end{equation}




By interpolating these inequalities via H\"older's inequalities, one gets the following fractional Hardy-Sobolev inequality.  

\begin{lemma}[Lemma 2.1 in \cite{Ghoussoub-Shaya}]  Assume that $0<\alpha<2$, $ 0 \le s \le \alpha<n$ and $2< p \le  {2_{\alpha}^*(s)}={\frac{2(n-s)}{n-{\alpha}}}$. Then, there exists a positive constant $C $  such that

\begin{equation} \label{fractional H-S-M inequality}
C (\int_{\Omega} \frac{|u|^p}{|x|^{s}}dx)^\frac{2}{p}   \leq  c(n,\alpha) \int_{\R^n}  \int_{\R^n} \frac{|u(x)- u(y)|^2}{|x-y|^{n+\alpha}} dxdy - \gamma \int_{\Omega}\frac{|u|^{2}}{|x|^{\alpha}} dx \quad  \text{ for } u \in \X,  
\end{equation}
 as long as $ \gamma <  \gamma_H(\alpha):=2^\alpha \frac{\Gamma^2(\frac{n+\alpha}{4})}{\Gamma^2(\frac{n-\alpha}{4})}.$

 \end{lemma}

Finally, we can define the general best Hardy-sobolev constant in the above inequality as

\begin{equation}
S_p := \inf_{{u \in \X \setminus \{0 \}}} \frac{c(n, \alpha) \int_{\R^n} \int_{\R^n} \frac{|u(x)- u(y)|^2}{|x-y|^{n+\alpha}}  dxdy - \gamma \int_{\Omega} \frac{|u|^2}{|x|^{\alpha}}dx}{ (\int_{\Omega} \frac{|u|^p}{|x|^{s}}dx)^\frac{2}{p}},
\end{equation}

where $2< p \le  {2_{\alpha}^*(s)}={\frac{2(n-s)}{n-{\alpha}}},$ and $ \gamma < \gamma_H(\alpha).$

Note that the frational Hardy inequality (\ref{Fractional Hardy Inequality}) asserts that $\X
$ is embedded in the weighted space $L^2(\Omega, |x|^{-\alpha})$ and that this embeding is continuous. If $\gamma < \gamma_H(\alpha)$, it follows from the fractional Hardy inequality (\ref{Fractional Hardy Inequality}) that 
$$ \|u\| := \left( c(n, \alpha) \int_{\R^n} \int_{\R^n} \frac{|u(x)- u(y)|^2}{|x-y|^{n+\alpha}}  dxdy - \gamma \int_{\Omega} \frac{|u|^2}{|x|^{\alpha}} dx \right)^\frac{1}{2} $$

is well-defined on $\X$. It also is equivalent to the norm $\|  \  . \ \|_{\X}.$ Thus, we can rewrite the functional $I_{\lambda,p}$ as

 \begin{equation*}
\begin{aligned}
I_{\lambda,p}(u)= \frac{1}{2} \| u\|^2  -\frac{\lambda}{q} \int_{\Omega}  f(x) |u|^q dx-\frac{1}{p}\int_{\Omega} g(x) \frac{|u|^{p}}{|x|^{s}} dx.
\end{aligned}
\end{equation*}

\section{Preliminary Results}

For $\lambda>0,$ we will consider the following Nehari minimization problem:

$$ \M = \inf \{ I_{\lambda,p}(u) : u \in \Ne \}, \  \text{ where } \mathcal{N} = \{ u \in \X :  \langle I'_{\lambda,p}(u),u \rangle = 0 \}.$$

Define

$$\fy (u):= \langle I'_{\lambda,p}(u),u \rangle=  \| u\|^2  - \lambda \int_{\Omega}  f(x) |u|^q dx -\int_{\Omega} g(x) \frac{|u|^{p}}{|x|^{s}} dx.$$  

So,  

\begin{equation}\label{Derivative of phi}
\langle \fy' (u),u  \rangle = 2 \| u\|^2  - \lambda \ q \int_{\Omega}  f(x) |u|^q dx - p \int_{\Omega} g(x) \frac{|u|^{p}}{|x|^{s}} dx.
\end{equation}  

Thus, for all $u \in \Ne ,$ we have  the following identities which will be used frequently in this paper.

\begin{equation}\label{formula for  all u in Nehari Manifold}
\| u\|^2  =  \lambda \int_{\Omega}  f(x) |u|^q dx+\int_{\Omega} g(x) \frac{|u|^{p}}{|x|^{s}} dx, 
\end{equation}

\begin{equation}\label{energy functional I_Lambda for nehari }
\begin{aligned}
I_{\lambda,p}(u)&= (\frac{1}{2}  -\frac{1}{q}) \| u\|^2 -(\frac{1}{p}-\frac{1}{q} ) \int_{\Omega} g(x) \frac{|u|^{p}}{|x|^{s}} dx\\
&= (\frac{1}{2}  -\frac{1}{p}) \| u\|^2 +(\frac{1}{p}-\frac{1}{q} ) \ \lambda \int_{\Omega}  f(x) |u|^q  dx,
\end{aligned}
\end{equation}

and

\begin{equation}\label{derivative of functional phi for nehari}
\begin{aligned}
 \langle \fy' (u),u  \rangle &= (2-q) \| u\|^2  - (p-q) \int_{\Omega} g(x) \frac{|u|^{p}}{|x|^{s}} dx\\ &= (2-p) \| u\|^2  - \lambda \ (q-p) \int_{\Omega}  f(x) |u|^q dx.
\end{aligned}
\end{equation}

Now we split $\Ne$ into three parts:

$$\mathcal{N^+} = \{ u \in \mathcal{N} :  \langle \phi'_{\lambda,p}(u),u \rangle > 0 \},$$
$$\mathcal{N}^0 = \{ u \in \mathcal{N} :  \langle \phi'_{\lambda,p}(u),u \rangle = 0 \},$$
$$\mathcal{N^-} = \{ u \in \mathcal{N} :  \langle \phi'_{\lambda,p}(u),u \rangle < 0 \}.$$

We first show that for $\lambda$ small enough, $\Ne^0$ is an empty set.
\begin{lemma}\label{Lemma - Nehari0 is an emptyset }
There exists a constant $\Lambda_1:= \Lambda_1(p) >0$ such that for  any $\lambda \in (0,\Lambda_1)$, we have $\Ne^0 = \emptyset.$
\end{lemma}

\begin{proof}
We deduce by contradiction. Suppose that there exists $u \in \X \setminus \{0\}$ such that $u \in \Ne^0,$ that is, $\langle \phi' (u),u  \rangle=0.$
We will consider the two following cases:\\

\textbf{Case 1:} $ \int_{\Omega}  f(x) |u|^q dx =0. $ Using  (\ref{derivative of functional phi for nehari}) and the fact that  $ \int_{\Omega}  f(x) |u|^q dx =0,$ we get   

$$0=\langle \fy' (u),u  \rangle = (2-p) \| u\|^2. $$

On the other hand, the assumption $p >2$ implies that $(2-p) \| u\|^2< 0$, which contradicts the last equality.\\

\textbf{Case 2:}  $ \int_{\Omega}  f(x) |u|^q dx \neq 0. $ It follows from (\ref{derivative of functional phi for nehari}) that 
\begin{equation}
\|u\|^2 = \lambda (\frac{p-q}{p-2}) \int_{\Omega} f(x) |u|^q dx.
\end{equation}

By the definition of $S_p$ and  the H\"older inequality, we get that 

\begin{equation}
\|u\|^2 \le \lambda (\frac{p-q}{p-2}) A \left(\int_{\Omega}  \frac{|u|^{p}}{|x|^{s}} dx \right)^{\frac{q}{p}} \le \lambda (\frac{p-q}{p-2}) A S_p^{-\frac{q}{2}} \|u\|^q,
\end{equation}
which yields

\begin{equation}
\|u\| \le \left(\lambda (\frac{p-q}{p-2}) A S_p^{-\frac{q}{2}} \right)^{\frac{1}{2-q}} ,
\end{equation}

where $A:= \left(\int_{\Omega} |x|^{\frac{sq}{p-q}} |f(x)|^{\frac{p}{p-q}} dx \right)^{\frac{p-q}{p}}.$  Define the functional $J_{\lambda,p} : \Ne\longrightarrow \R$ as

$$J_{\lambda,p}(u) = \left( \frac{p-2}{2-q}\right) \left( \frac{2-q}{p-q}\right)^{\frac{p-1}{p-2}} \left(       \frac{\|u\|^{2(p-1)}}{\int_{\Omega} g(x) \frac{|u|^{p}}{|x|^{s}}  dx } \right)^{\frac{1}{p-2}} - \lambda \int_{\Omega} f(x) |u|^q dx. $$

We claim that $J_\lambda(u) =0$ for all $u \in \Ne^0.$
Indeed, by (\ref{derivative of functional phi for nehari}),  we have 

\begin{equation} 
\hbox{$  \int_{\Omega} g(x) \frac{|u|^{p}}{|x|^{s}} dx = \left(\frac{2-q}{p-q}\right) \| u\|^2 \quad     \text{and} \quad \lambda \int_{\Omega} f(x) |u|^q dx = \left( \frac{p-2}{2-q}\right) \|u\|^2  \quad \text{ for all } u \in \Ne^0.$}
\end{equation}

Thus,

\begin{equation} \label{J_lambda (u) = 0 for u in Nehari 0}
\begin{aligned}
J_{\lambda,p}(u) &= \left( \frac{p-2}{2-q}\right) \left( \frac{2-q}{p-q}\right)^{\frac{p-1}{p-2}} \left(       \frac{\|u\|^{2(p-1)}}{\left(\frac{2-q}{p-q}\right) \| u\|^2 } \right)^{\frac{1}{p-2}} - \lambda \int_{\Omega} f(x) |u|^q dx \\
&= \left( \frac{p-2}{2-q}\right) \|u\|^2 - \lambda \int_{\Omega} f(x) |u|^q dx= 0 \qquad \text{ for all } u \in \Ne^0.
\end{aligned}
\end{equation}


Let $C(p,q) := \left( \frac{p-2}{2-q}\right) \left( \frac{2-q}{p-q}\right)^{\frac{p-1}{p-2}} $. By  H\"older's inequality and the definition of $S_p,$ we obtain

\begin{equation}
\begin{aligned}
J_{\lambda,p}(u) &\ge C(p,q) \left(       \frac{\|u\|^{2(p-1)}}{\int_{\Omega} g(x) \frac{|u|^{p}}{|x|^{s}} dx} \right)^{\frac{1}{p-2}} - \lambda A \left(\int_{\Omega}  \frac{|u|^{p}}{|x|^{s}} dx \right)^{\frac{q}{p}}\\
& \ge  \left(\int_{\Omega}  \frac{|u|^{p}}{|x|^{s}} dx \right)^{\frac{q}{p}}  \left[ C(p,q) S_p^{\frac{p-1}{p-2}}  \|g\|_\infty^{\frac{1}{2-p}}              \left(\int_{\Omega}  \frac{|u|^{p}}{|x|^{s}} dx \right)^{\frac{1-q}{p}} - \lambda A\right]\\
& \ge  \left(\int_{\Omega}  \frac{|u|^{p}}{|x|^{s}} dx \right)^{\frac{q}{p}}  \left[ C(p,q) S_p^{\frac{p-1}{p-2}+ \frac{q-1}{2-q}} \left( \frac{\lambda A (p-q)}{p-2}\right)^{\frac{1-q}{2-q}}    \|g\|_\infty^{\frac{1}{2-p}}    - \lambda A\right].
\end{aligned}
\end{equation}

Thus, we get that $J_{\lambda,p}(u) >0$ 
for $\lambda$ sufficiently small. Therefore, there exists $\Lambda_1:=\Lambda_1(p)>0$ such that  $J_{\lambda,p}(u) >0$ for all  $\lambda \in (0,\Lambda_1)$ and $u \in \Ne^0.$ This contradicts (\ref{J_lambda (u) = 0 for u in Nehari 0}) and completes the proof. 

\end{proof}

\begin{lemma}\label{Lemma - for u in N+, lambda int f >0}
%
 If $u \in \Ne^+ \setminus \{0\},$ then $\int_{\Omega}  f(x) |u|^q dx >0.$

\end{lemma}

\begin{proof}
Since $u \in \Ne^+ \setminus \{0\},$ we have $\langle \fy'(u), u \rangle > 0.$ It then follows from (\ref{derivative of functional phi for nehari}) that  

$$\left(\frac{2-q}{p-q } \right) \| u\|^2  >\int_{\Omega} g(x) \frac{|u|^{p}}{|x|^{s}} dx. $$

By  (\ref{formula for  all u in Nehari Manifold}) and the last inequality, we get that 

\begin{align*}
\lambda \int_{\Omega}  f(x) |u|^q dx &= \| u\|^2  - \int_{\Omega} g(x) \frac{|u|^{p}}{|x|^{s}} dx \\
& > \| u\|^2  - \left(\frac{2-q}{p-q } \right) \| u\|^2\\
&>  \left(\frac{p-2}{p-q } \right) \| u\|^2 >0.
\end{align*}


\end{proof}

From Lemma \ref{Lemma - Nehari0 is an emptyset }, we deduce that $\Ne = \Ne^+ \cup \Ne^-$ for any $\lambda \in (0,\Lambda_1).$ Define

$$\M^+ := \inf\limits_{\Ne^+} I_{\lambda,p}(u) \quad \text{ and } \quad \M^- := \inf\limits_{\Ne^-} I_{\lambda,p}(u).$$

\begin{lemma}
For any $\lambda \in  (0,\Lambda_1)$, the minimizers on $\Ne $ are critical points for $I_{\lambda,p}$ in $(\X)',$ where $(\X)'$ is dual space of $\X.$

\end{lemma}

\begin{proof}

Suppose that $\bar{u}$ is a local minimum for $I_{\lambda,p}.$ Thus, it satisfies the following minimization problem:

$$\min\limits_{u \in \X}  \left\{ I_{\lambda,p}(u) : \fy = \langle I_{\lambda,p}'(u) , u \rangle =0 \right\}, $$

which gives

$$ I_{\lambda,p}(\bar{u}) = \min\limits_{u \in \X} I_{\lambda,p}(u)  \quad  \text{ and } \quad \fy(\bar{u}) =\langle I_{\lambda,p}'(\bar{u}), \bar{u} \rangle=0. $$

It follows from the theorem of Lagrange multiplies that there exists $\theta$ such that  $I_{\lambda,p}'(\bar{u}) = \theta \fy'(\bar{u})$ in $(\X)'.$ So, we have 

$$0 = \langle I_{\lambda,p}'(\bar{u}), \bar{u} \rangle  = \langle \theta \fy'(\bar{u}), u \rangle =  \theta \langle  \fy'(\bar{u}), \bar{u} \rangle.$$

Thus, $$\text{ either } \theta=0 \quad \text{ or } \quad \langle  \fy'(\bar{u}), \bar{u} \rangle =0.$$

By Lemma \ref{Lemma - Nehari0 is an emptyset }, we get that  $\langle  \fy'(u), u \rangle \neq 0$ for $u \neq 0.$ Therefore $\theta =0.$ Thus, we obtain   $I_{\lambda,p}'(\bar{u}) = \theta \fy'(\bar{u}) =0$ in $(\X)',$ that is, $\bar{u}$ is a critical point for $I_{\lambda,p}$ in $(\X)'.$

\end{proof}

\begin{lemma}\label{Lemma - existence of t+ and t-}
Let  $\Lambda_2:=\Lambda_2(p) = \left( \frac{p-2}{p-q } \right)    \left( \frac{2-q}{p-q }  \right)^{\frac{2-q}{p-2}} S_p^{\frac{p-q}{p-2}} A^{-1} \|g\|^{\frac{q-2}{p-2}}_{\infty}.$ Then, for all $u \in \X \setminus \{0\}$ and $\lambda \in(0, \Lambda_2),$ there exist unique $t^+(u)$ and  $t^-(u)$ such that 

\begin{enumerate}
\item $0 \le t^+(u) < t_{\max}< t^-(u).$
\item $t^-(u) u\in \Ne^-$ and $ t^+(u) u \in \Ne^+.$
\item   $ I_{\lambda,p}(t^-(u) u) = \max\limits_{t > t_{\max}} I_{\lambda,p}(tu)  $ and $ I_{\lambda,p}(t^+(u) u) = \min\limits_{0 \le t \le t^-(u)} I_{\lambda,p}(tu)  .$
\item  $ \Ne^- = \left\{  u \in \X \setminus \{0\}: t^- (\frac{u}{\|u\|}) = \|u\|       \right\},$

\end{enumerate}
where $t_{\text{max}}:=\left( \frac{2-q}{p-q } \frac{\| u\|^2 }{\int_{\Omega} g(x) \frac{|u|^{p}}{|x|^{s}} dx}\right) ^{\frac{1}{p-2}}.$ Moreover, $t^+(u) >0$ if and only if $\int_\Omega f(x) |u|^q dx >0.$

\end{lemma}

\begin{proof}

For $t\ge0,$ define 

$$h(t) = t^{2-q} \|u\|^2 - t^{p-q} \int_{\Omega} g(x) \frac{|u|^{p}}{|x|^{s}}dx.$$

Straightforuard  computations yield that $h(0)=0,$ $\lim\limits_{ t \to \infty} h(t) = -\infty,$ $h'(t_{\max}) = 0,$ and $h(t)$ is attained its maximum at $t_{\max}.$ In addition, $h(t)$ is increasing for $t \in [0,t_{\max})$ and decreasing for $t \in (t_{\max} , \infty).$ So, we have

$$h(t_{\max}) =  \left(\frac{p-2}{p-q }\right) \left( \frac{2-q}{p-q } \right)^{\frac{2-q}{p-2}} \|u\|^q    \left( \frac{\| u\|^p }{\int_{\Omega} g(x) \frac{|u|^{p}}{|x|^{s}} dx}\right) ^{\frac{2-q}{p-2}}.   $$

By H\"older's inequality and the definition of $S_p,$ we obtain

\begin{equation} \label{lower bound for h(t-max)}
h(t_{\max}) \ge    \left(\frac{p-2}{p-q }\right) \left( \frac{2-q}{p-q } \right)^{\frac{2-q}{p-2}}  \|u\|^q  \|g\|_\infty^{\frac{q-2}{p-2}} S_p^{\frac{p(2-q)}{2(p-2)}}.
\end{equation}

We will now consider the two following cases:\\

\textbf{Case 1:}  $\int_\Omega f(x) |u|^q dx \le 0.$
In this case, there exists a unique $t^- := t^-(u)> t_{\max}$ such that

\begin{equation} \label{h(t-) and h'(t-)<0}
 h(t^-) = \lambda \int_\Omega f(x) |u|^q dx \ \text{ and } \ h'(t^-) <0.
\end{equation}

We claim that $t^- u \in \Ne^-.$ Indeed, clearly  $t^- u \in \X,$ and 
(\ref{h(t-) and h'(t-)<0}) implies that
\begin{align*}
\langle  I_{\lambda,p}'(t^-u),  t^-u \rangle &=  \| t^-u\|^2  - \lambda \int_{\Omega}  f(x) |t^- u|^q dx -\int_{\Omega} g(x) \frac{|t^- u|^{p}}{|x|^{s}} dx\\
& = (t^-)^q \left(      (t^-)^{2-q} \|u\|^2 - (t^-)^{p-q} \int_{\Omega} g(x) \frac{|u|^{p}}{|x|^{s}}dx - \lambda \int_{\Omega}  f(x) | u|^q dx \right)\\
&=  (t^-)^q \left( h(t^-) - \lambda \int_{\Omega}  f(x) | u|^q dx \right) =0,
\end{align*}

and

\begin{align*}
\langle  \fy'(t^-u),  t^-u \rangle &=  2 \| t^-u\|^2  - \lambda q \int_{\Omega}  f(x) |t^- u|^q dx - p\int_{\Omega} g(x) \frac{|t^- u|^{p}}{|x|^{s}} dx\\
& =  (2-q) \| t^-u\|^2  - (p-q) \int_{\Omega} g(x) \frac{|t^- u|^{p}}{|x|^{s}} dx\\
& =(t^-)^{q+1} \left(       (2-q) (t^-)^{2-q-1} \| u\|^2  - (p-q) (t^-)^{p-q-1} \int_{\Omega} g(x) \frac{| u|^{p}}{|x|^{s}} dx \right)\\
& = (t^-)^{q+1} h'(t^-) <0.
\end{align*}
This proves the claim, and we have that  $t^- u \in \Ne^-.$  In order to prove that $I_{\lambda,p}(t^-u) = \max\limits_{t\ge t_{\max}} I_{\lambda,p}(tu),$ we need to show that 

\begin{equation}
\frac{d}{dt} I_{\lambda,p}(t^-u) =0 \ \text{ and } \ \frac{d^2}{dt^2} I_{\lambda,p}(tu) <0 \quad \text { for } t > t_{\max}.
\end{equation}

It follows from (\ref{h(t-) and h'(t-)<0}) that 

\begin{align*}
\frac{d}{dt} I_{\lambda,p}(t^-u) &=  t^- \| u\|^2  - \lambda (t^-)^{q-1}\int_{\Omega}  f(x) | u|^q dx - (t^-)^{p-1}\int_{\Omega} g(x) \frac{| u|^{p}}{|x|^{s}} dx\\
&=  (t^-)^{q-1} \left( h(t^-) - \lambda \int_{\Omega}  f(x) | u|^q  dx\right) =0.
\end{align*}

We also have 

\begin{align*}
t^2 \frac{d^2}{dt^2} I_{\lambda,p}(tu) &=  \| t^-u\|^2  - \lambda (q-1) \int_{\Omega}  f(x) |t^- u|^q dx - (p-1)\int_{\Omega} g(x) \frac{|t^- u|^{p}}{|x|^{s}} dx\\
& =(t^-)^{q+1} \left(       (2-q) (t^-)^{2-q-1} \| u\|^2  - (p-q) (t^-)^{p-q-1} \int_{\Omega} g(x) \frac{| u|^{p}}{|x|^{s}} dx \right)\\
& = (t^-)^{q+1} h'(t^-)<0 \quad \text{ for all } t > t_{\max}.
\end{align*}

\textbf{Case 2:}  $\int_\Omega f(x) |u|^q dx > 0.$ Using H\"older's inequality  and  (\ref{lower bound for h(t-max)}), we have 

\begin{align*}
0 = h(0) &< \int_\Omega f(x) |u|^q dx \le \lambda A S_p^{-\frac{q}{2}} \|u\|^q \\
& \le \left(\frac{p-2}{p-q }\right) \left( \frac{2-q}{p-q } \right)^{\frac{2-q}{p-2}}   \|g\|_\infty^{\frac{q-2}{p-2}} S_p^{\frac{p(2-q)}{2(p-2)}}   \|u\|^q\\
& \le h(t_{\max})  \quad \text{ for }  0<\lambda < \Lambda_2. 
\end{align*}

Using the assumption  $\int_\Omega f(x) |u|^q dx > 0$ and the fact that $h(t_{\max}) >0,$ we get  that there exist unique $t^+:= t^+(u)$ and $t^-:=t^-(u)$ such that $t^+<t_{\max}<t^-,$ and
\begin{equation*}
h(t^-) = \lambda \int_\Omega f(x) |u|^q dx = h(t^+) \ \text{ and } \  h'(t^-) < 0 < h'(t^+).
\end{equation*}

\end{proof}

\begin{lemma}\label{Lemma functional I is coercive and bounded below} The following hold.
\begin{enumerate}
\item $\M \le \M^+ <0.$
\item Let $\Lambda_3:=\Lambda_3(p) = \frac{p-2}{p-q}.$ Then, the functional $I_{\lambda,p}$ is coercive and bounded below on $\Ne$ for any $\lambda \in (0, \Lambda_3].$
\end{enumerate}

\end{lemma}

\begin{proof} By (\ref{energy functional I_Lambda for nehari }), for any $ u \in \Ne,$ we have 

$$I_{\lambda,p}(u)= (\frac{1}{2}  -\frac{1}{p}) \| u\|^2 +(\frac{1}{p}-\frac{1}{q} ) \ \lambda \int_{\Omega}  f(x) |u|^q  dx.$$

1. Suppose that $u \in \Ne^+.$ It follows from (\ref{derivative of functional phi for nehari}) that 

\begin{align*}
I_{\lambda,p}(u) &< \left( \frac{p-2}{2p}  \right)\left( \frac{p-q}{p-2}  \right) \lambda \int_{\Omega}  f(x) |u|^q  dx + \left( \frac{q-p}{pq}  \right) \lambda \int_{\Omega}  f(x) |u|^q  dx\\
& = - \frac{(p-q)(2-q)}{2pq}  \lambda \int_{\Omega}  f(x) |u|^q  dx.
\end{align*}   

By Lemma \ref{Lemma - for u in N+, lambda int f >0}, we have $\int_{\Omega}  f(x) |u|^q  dx>0.$ Thus,

$$ I_{\lambda,p}(u) <  - \frac{(p-q)(2-q)}{2pq}  \lambda \int_{\Omega}  f(x) |u|^q  dx < 0, $$

which yields

$$ \M \le \M^+ <0.$$

2. Using H\"older  and Young's  inequality, we get that 

$$I_{\lambda,p}(u) \ge \frac{1}{2p} \left( (p-2) - \lambda (p-q) \right) \|u\|^2 - \lambda \left(\frac{p-q}{pq} \right)  \left(\frac{2-q}{2} \right) \left( A S_p^{-\frac{q}{2}} \right)^{\frac{2}{2-q}}.$$

Since $ 0< \lambda < \frac{p-2}{p-q},$ the functional  $I_{\lambda,p}$ is coercive and bounded below on $\Ne,$ and we have  

$$   I_{\lambda,p} (u) \ge  - \lambda \left(\frac{p-q}{pq} \right)  \left(\frac{2-q}{2} \right) \left( A S_p^{-\frac{q}{2}} \right)^{\frac{2}{2-q}}.$$

\end{proof}

\begin{lemma}\label{Lemma - existence of diffrentiable function sigma on Ne}
For each $u \in \Ne \setminus \{0\},$ there exist $\epsilon>0$ and a differentiable function $\sigma :  B(0,\epsilon) \subset \X \longrightarrow \R^+$ such that $\sigma(0)=1,$ $\sigma(v) (u-v) \in \Ne,$ and

\begin{equation} \label{derivative of sigma at 0}
\langle   \sigma'(0),v \rangle = \frac{2 c(n, \alpha) \langle  u , v \rangle_{\X}  - 2 \gamma \int_\Omega \frac{uv}{|x|^\alpha} dx - q \lambda\int_{\Omega}  f(x) |u|^{q-1}uv dx- p\int_{\Omega} g(x) \frac{|u|^{p-2}uv}{|x|^{s}}  dx       }{ (2-q) \left( c(n, \alpha) \int_{\R^n} \int_{\R^n} \frac{|u(x)- u(y)|^2}{|x-y|^{n+\alpha}}  dxdy -  \gamma \int_\Omega \frac{|u|^2}{|x|^\alpha} dx \right) - (p-q) \int_{\Omega} g(x) \frac{|u|^p}{|x|^{s}}  dx },
\end{equation}
 
for all $v \in\X.$ Here $B(0,\epsilon) := \{ u \in \X: \|u\| < \epsilon \}.$
\end{lemma}

\begin{proof}
For $u \in \Ne,$ define $G:\R \times \X \longrightarrow \R$ as $$G(t,v) = \langle I_{\lambda,p}'(t(u-v)) , t(u-v) \rangle.$$ 

So, we have

\begin{align*}
G(t,v) &= t^2 c(n, \alpha) \int_{\R^n} \int_{\R^n} \frac{|u(x)- u(y)|^2}{|x-y|^{n+\alpha}}  dxdy -  t^2\gamma \int_\Omega \frac{|u|^2}{|x|^\alpha} dx \\
& - \lambda t^q \int_\Omega f(x) |u-v|^q dx -t^p \int_{\Omega} g(x) \frac{|u|^p}{|x|^{s}}  dx, 
\end{align*}

and  $$ G(1,0) = \langle I_{\lambda,p}'(u) , u \rangle. $$ 

By Lemma \ref{Lemma - Nehari0 is an emptyset }, we obtain $\frac{d}{dt} G(1,0) \neq 0,$ that is,

\begin{align*}
0 \neq \frac{d}{dt} G(1,0) &=  2  t c(n, \alpha) \int_{\R^n} \int_{\R^n} \frac{|u(x)- u(y)|^2}{|x-y|^{n+\alpha}}  dxdy - 2 t \gamma \int_\Omega \frac{|u|^2}{|x|^\alpha} dx\\
&- q \lambda t^{q-1}\int_{\Omega}  f(x) |u|^q dx - pt^{p-1}\int_{\Omega} g(x) \frac{|u|^p}{|x|^{s}} dx\biggr\rvert_{t=1}\\
&  =  (2-q) \|u\|^2 -(p-q) \int_{\Omega} g(x) \frac{|u|^p}{|x|^{s}} dx.
\end{align*}

According to the implicit function theorem, there exist $\epsilon>0$ and a differentiable function $\sigma : B(0,\epsilon) \longrightarrow \R$ such that $\sigma(0)=1,$ and

\begin{equation*} 
\langle   \sigma'(0),v \rangle = \frac{2 c(n, \alpha) \langle  u , v \rangle_{\X}  - 2 \gamma \int_\Omega \frac{uv}{|x|^\alpha} dx - q \lambda\int_{\Omega}  f(x) |u|^{q-1}uv dx- p\int_{\Omega} g(x) \frac{|u|^{p-2}uv}{|x|^{s}}  dx       }{ (2-q) \left( c(n, \alpha) \int_{\R^n} \int_{\R^n} \frac{|u(x)- u(y)|^2}{|x-y|^{n+\alpha}}  dxdy -  \gamma \int_\Omega \frac{|u|^2}{|x|^\alpha} dx \right) - (p-q) \int_{\Omega} g(x) \frac{|u|^p}{|x|^{s}}  dx }.
\end{equation*}

Moreover, we have $G(\sigma(v),v) =0,$ for all $v \in B(0,\epsilon),$ which implies that  
$\langle I_{\lambda,p}'(\sigma(v)(u-v)), \sigma(v)(u-v))  \rangle =0,$  that is,  $\sigma(v)(u-v) \in \Ne.$

\end{proof}

\begin{lemma}\label{Lemma - existence of diffrentiable function sigma on Ne^-}
For each $u \in \Ne^- \setminus \{0\},$ there exist $\epsilon>o$ and a differentiable function $\sigma_- :  B(0,\epsilon) \subset \X \longrightarrow \R^+$ such that $\sigma_-(0)=1,$ $\sigma_-(v) (u-v) \in \Ne^-,$ and

\begin{equation} \label{derivative of sigma_- at 0}
\langle   {\sigma_-}'(0),v \rangle = \frac{2 c(n, \alpha) \langle  u , v \rangle_{\X} - 2 \gamma \int_\Omega \frac{uv}{|x|^\alpha} dx - q \lambda\int_{\Omega}  f(x) |u|^{q-1}uv dx- p\int_{\Omega} g(x) \frac{|u|^{p-2}uv}{|x|^{s}}  dx       }{ (2-q) \left(c(n, \alpha) \int_{\R^n} \int_{\R^n} \frac{|u(x)- u(y)|^2}{|x-y|^{n+\alpha}}  dxdy -  \gamma \int_\Omega \frac{|u|^2}{|x|^\alpha} dx \right) - (p-q) \int_{\Omega} g(x) \frac{|u|^p}{|x|^{s}}  dx },
\end{equation}
 
for all $v \in\X.$ 
\end{lemma}

\begin{proof}
Following the proof of Lemma \ref{Lemma - existence of diffrentiable function sigma on Ne}, we get that there exist $\epsilon>0$ and a differentiable function $\sigma_- : B(0,\epsilon) \longrightarrow \R$ such that $\sigma_-(0)=1$ and $\sigma_-(v)(u-v) \in \Ne, $  for all $v \in B(0,\epsilon).$ Since $u \in \Ne^-,$ we have 

$$ \langle \phi' (u),u  \rangle = (2-q) \| u\|^2  - (p-q) \int_{\Omega} g(x) \frac{|u|^{p}}{|x|^{s}} dx <0. $$

It then follows from the continuity of $\phi' $ and $\sigma_-$ that

$$ \langle \phi' (\sigma_-(v)(u-v)),\sigma_-(v)(u-v)  \rangle = (2-q) \| \sigma_-(v)(u-v)\|^2  - (p-q) \int_{\Omega} g(x) \frac{|\sigma_-(v)(u-v)|^{p}}{|x|^{s}} dx . $$

Therefore, for $\epsilon>0$  small enough, we get that $ \sigma_- (v)(u-v) \in \Ne^-.$
\end{proof}

\begin{proposition}\label{Proposition - existence of minimizing sequence, using Ekeland principle}

Let $\Lambda= \Lambda(p) := \min\{ \Lambda_1, \Lambda_2, \Lambda_3    \}.$ Then, for any $\lambda \in (0,\Lambda),$ the following hold.   

\begin{enumerate}
 \item There exists a minimizing sequence $(u_k)_{k \in \N} \subset \Ne$ for $I_{\lambda,p}(u)$ such that 
 
 \begin{itemize}
  \item $I_{\lambda,p}(u_k) = \M +o(1).$
  \item $I_{\lambda,p}'(u_k) = o(1) \quad \text{ in }  (\X)'.$
 \end{itemize} 
  \item There exists a minimizing sequence $(u_k)_{k \in \N} \subset \Ne^-$ for $I_{\lambda,p}(u)$ such that 
 
 \begin{itemize}
  \item $I_{\lambda,p}(u_k) = \M^- +o(1).$
  \item $I_{\lambda,p}'(u_k) = o(1) \qquad \text{ in }  (\X)'.$
 \end{itemize}

 \end{enumerate} 
\end{proposition}

\begin{proof}

 It follows from Lemma \ref{Lemma functional I is coercive and bounded below} that $I_{\lambda,p}(u)$ is coercive and bounded below. Then,\\

1. The Ekeland variational principle implies that there exists a minimizing sequence $(u_k)_{k \in \N}$ such that 
\begin{equation}\label{Ekeland principle -I_lambda < M + frac(1)- (k)  and I_lambda(u_k) < I_lambda(u) + frac(1)- (k) || u-u_k|| on Nehari}
I_{\lambda,p}(u_k) < \inf\limits_{\Ne} I_{\lambda,p}(u) + \frac{1}{k} = \M + \frac{1}{k} \ \text{ and } \ I_{\lambda,p}(u_k) < I_{\lambda,p}(u) + \frac{1}{k} \|u-u_k\| \quad \text{ for all } u \in \Ne.
\end{equation} 
 For $k$ large enough, we use Lemma \ref{Lemma functional I is coercive and bounded below}, (\ref{energy functional I_Lambda for nehari }) and (\ref{Ekeland principle -I_lambda < M + frac(1)- (k)  and I_lambda(u_k) < I_lambda(u) + frac(1)- (k) || u-u_k|| on Nehari}) to get 

\begin{equation} \label{I_lambda < M <0}
\begin{aligned}
 (\frac{1}{p}-\frac{1}{q} ) \ \lambda \int_{\Omega}  f(x) |u_k|^q  dx &\le I_{\lambda,p}(u_k) = (\frac{1}{2}  -\frac{1}{p}) \| u_k\|^2 +(\frac{1}{p}-\frac{1}{q} ) \ \lambda \int_{\Omega}  f(x) |u_k|^q  dx\\
 & < \M +\frac{1}{n}\\
 &<\M\\
  &<\frac{\M}{2}\\
 &<0.
\end{aligned}
\end{equation} 

Therefore,

\begin{equation*}
-\frac{\M}{2} \frac{pq}{p-q} < \lambda \int_{\Omega}  f(x) |u_k|^q  dx \le  A S_p^{-\frac{q}{2}} \|u_k\|^q,
\end{equation*}

which yields $u_k \neq 0,$ for all $k \in \N.$ On the other hand, from (\ref{I_lambda < M <0}) and H\"older's inequality, we deduce that 

$$ \| u_k\|^2  < \frac{2 \lambda (p-q)}{q(p-2)} A S_p^{-\frac{q}{2}} \|w_k\|^q.$$

Hence, 

\begin{equation}\label{Lower and Upper bound for ||u_k||}
\left(-\frac{\M}{2} \frac{pq}{p-q} A^{-1} S_p^{\frac{q}{2}}\right)^{\frac{1}{q}}< \| u_k\|  < \left(\frac{2 \lambda (p-q)}{q(p-2)} A S_p^{-\frac{q}{2}}\right)^{\frac{1}{2-q}}.
\end{equation}

In order to finalize the proof, it is sufficient to show that 

\begin{equation}\label{claim || I' || go to 0}
\|I_{\lambda,p}'(u_k)\|_{(\X)'} \to 0 \quad \text{ as } {k \to \infty}.
\end{equation}

Indeed,  it follows from Lemma \ref{Lemma - existence of diffrentiable function sigma on Ne} that there exists a differentiable function $\sigma_k : B_k(0,\epsilon_k) \longrightarrow \R_+,$ for some $\epsilon_k,$ such that 
$$ \sigma_k (u_k-u) \in \Ne \quad \text{ for all } k \in \N,$$
where $B_k(0,\epsilon_k) := \{ u \in \X: \|u\| < \epsilon_k \}.$ Choose $0<\rho <\epsilon_k,$ and for any $u \in \X \setminus \{ 0 \},$ define $u_\rho:= \frac{\rho u}{\|u\|}$ and $\eta_\rho:= \sigma_k(\rho_k)(u_k-u).$ 
Using the fact that $\eta_\rho \in \Ne,$ and also $(\ref{Ekeland principle -I_lambda < M + frac(1)- (k)  and I_lambda(u_k) < I_lambda(u) + frac(1)- (k) || u-u_k|| on Nehari})_2,$ we get that 
$I_{\lambda,p}(u_k) <  I_{\lambda,p}(\eta_\rho) + \frac{1}{k} \|\eta_\rho - u_k\|,$
which means 

\begin{equation}
I_{\lambda,p}(\eta_\rho) -I_{\lambda,p}(u_k)   > - \frac{1}{k} \|\eta_\rho - u_k\|.
\end{equation}

Now we apply the mean value theorem to the left hand-side of the last inequality to deduce 

$$I_{\lambda,p}(\eta_\rho) -I_{\lambda,p}(u_k)  = \langle I_{\lambda,p}'(u_k), \eta_\rho -u_k  \rangle + o(\|\eta_\rho -u_k\|).$$

Thus,

\begin{equation} \label{Applying mean value on ekeland results}
\langle I_{\lambda,p}'(u_k), \eta_\rho -u_k  \rangle + o(\|\eta_k -u_k\|) \ge-\frac{1}{k} \|\eta_\rho - u_k\|.
\end{equation}

Regarding the first term in (\ref{Applying mean value on ekeland results}), we have that
\begin{equation*}
\begin{aligned}
\langle I_{\lambda,p}'(u_k), \eta_\rho -u_k  \rangle& = \langle I_{\lambda,p}'(u_k), \sigma_k(u_\rho)(u_k-u_\rho) - u_k \rangle\\
& = \langle I_{\lambda,p}'(u_k), \sigma_k(u_\rho)(u_k-u_\rho) +(u_\rho- u_k) - u_\rho \rangle\\
& =\langle I_{\lambda,p}'(u_k), - u_\rho \rangle+\langle I_{\lambda,p}'(u_k), (\sigma_k(u_\rho)-1)(u_k-u_\rho)  \rangle.
\end{aligned}
\end{equation*}

Therefore,

\begin{equation}
\langle I_{\lambda,p}'(u_k), - u_\rho \rangle+\langle I_{\lambda,p}'(u_k), (\sigma_k(u_\rho)-1)(u_k-u_\rho)  \rangle \ge-\frac{1}{k} \|\eta_\rho - u_k\|.
\end{equation}
 
By the definition of $u_\rho$ and $\eta_\rho,$ we obtain
\begin{equation*}
- \rho \langle I_{\lambda,p}'(u_k), \frac{u}{\|u\|} \rangle +(\sigma_k(u_\rho)-1) \langle I_{\lambda,p}'(u_k)-I_{\lambda,p}'(\eta_\rho), (u_k - u_\rho) \rangle \ge-\frac{1}{k} \|\eta_\rho - u_k\|+ o(\|\eta_\rho - u_k\|).
\end{equation*} 

The last inequality implies that 
\begin{equation}\label{estimate for < I',u/||u|| >}
 \langle I_{\lambda,p}'(u_k), \frac{u}{\|u\|}\rangle  \le \frac{1}{\rho k} \|\eta_\rho - u_k\|+\frac{o(\|\eta_\rho - u_k\|)}{\rho} + \frac{(\sigma_k(u_\rho)-1)}{\rho} \langle I_{\lambda,p}'(u_k)-I_{\lambda,p}'(\eta_\rho), (u_k - u_\rho) \rangle .
\end{equation} 

Note that from Lemma \ref{Lemma - existence of diffrentiable function sigma on Ne}, it follows 

$$\lim\limits_{\rho \to 0} \frac{|(\sigma_k(u_\rho)-1)|}{\rho} = \frac{| \langle \sigma'(0) , u_\rho \rangle |}{\rho} \le \| \sigma'(0) \|,$$

and also simple computations yield  

\begin{align*}
\| \eta_\rho  - u_k \| &= \|\sigma_k(u_\rho) (u_k - u_\rho) u_k\|\\
&= \|(\sigma_k(u_\rho) -1) u_k - u_\rho \sigma_k(u_\rho)\|\\
& \le \|(\sigma_k(u_\rho) -1) u_k \| +\| u_\rho \sigma_k(u_\rho)\|\\
&= |\sigma_k(u_\rho) -1| \|u_k\| + |\sigma_k(u_\rho)| \rho.
\end{align*}

Using the last two identities, and (\ref{estimate for < I',u/||u|| >}), we then get that

\begin{equation*}
 \langle I_{\lambda,p}'(u_k), \frac{u}{\|u\|}\rangle  \le \frac{1}{k} |\sigma_k (u_\rho)|+\frac{1}{k} \frac{|\sigma_k (u_\rho) -1|}{\rho} \|u_k\| + \frac{(\sigma_k(u_\rho)-1)}{\rho} \langle I_{\lambda,p}'(u_k)-I_{\lambda,p}'(\eta_\rho), (u_k - u_\rho) \rangle .
\end{equation*} 

Taking $\sigma \to 0$ in the last inequality for a fixed $k$, and using (\ref{Lower and Upper bound for ||u_k||}), we obtain that there exists a constant $C>0$ (independent of $\rho$) such that 

$$ \langle I_{\lambda,p}'(u_k), \frac{u}{\|u\|}\rangle  \le \frac{C}{k}(1+ \|\sigma'(0)\|) \quad \text{ as } \sigma \to 0.    $$

In order to complete the proof of (\ref{claim || I' || go to 0}), we only need to show that $\|\sigma'(0)\|$ is uniformly bounded in $k.$ It follows from (\ref{derivative of sigma at 0}) and H\"older's inequality that there exists a constant $c>0$ such that 

$$ \langle \sigma'(0),v   \rangle \le \frac{c}{\left| (2-q) \|u_k\|^2  - (p-q) \int_{\Omega} g(x) \frac{|u_k|^p}{|x|^{s}}  dx \right|}.  $$

It remains to prove that there exists a constant $\bar{c}>0$ such that 

\begin{equation}\label{The denominator of sigma's bound >0}
\left| (2-q) \|u_k\|^2 - (p-q) \int_{\Omega} g(x) \frac{|u_k|^p}{|x|^{s}} dx \right| > \bar{c} \quad \text{for } n \text{ large enough}.
\end{equation}
We deduce by contradiction. Suppose that there exists a sub-sequence $(u_k)_{k \in \N}$ such that 

\begin{equation}\label{The denominator of sigma's bound = o(1)[contradiction] }
(2-q) \|u_k\|^2 - (p-q) \int_{\Omega} g(x) \frac{|u_k|^p}{|x|^{s}} dx =o(1) \quad \text{ as } k \to \infty.
\end{equation} 
Then, (\ref{Lower and Upper bound for ||u_k||}) and (\ref{The denominator of sigma's bound = o(1)[contradiction] }) yield

\begin{align*}
(p-q) \int_{\Omega} g(x) \frac{|u_k|^p}{|x|^{s}} dx &= (2-q) \|u_k\|^2 +o(1)\\
& \ge (2-q) \left(-\frac{\M}{2} \frac{pq}{p-q} A^{-1} S_p^{\frac{q}{2}}\right)^{\frac{1}{q}} +o(1) \quad \text{ as } k \to \infty,
\end{align*}

which implies that there exists a constant $C_1>0$ such that 

\begin{equation}\label{Hardy-Sobolev term >=C1>0}
\int_{\Omega} g(x) \frac{|u_k|^p}{|x|^{s}} dx \ge C_1>0.
\end{equation}

In addition, by (\ref{The denominator of sigma's bound = o(1)[contradiction] }) and the fact that $(u_k)_{k \in \N} \in \Ne,$ we have 

\begin{align*}
\lambda  \int_{\Omega} f(x) |u_k|^q dx &= \|u_k\|^2 - \int_{\Omega} g(x) \frac{|u_k|^p}{|x|^{s}} dx \\
& = \|u_k\|^2 - \frac{2-q}{p-q} \|u_k\|^2 +o(1)\\
& =  \frac{p-2}{p-q} \|u_k\|^2+o(1) \quad \text{ as } k \to \infty.
\end{align*}

Hence,

\begin{align*}
\|u_k\| &= \left(\lambda (\frac{p-q}{p-2}) \int_{\Omega} f(x) |u_k|^q dx\right)^{\frac{1}{2}} + o(1)\\
& \le \left(\lambda (\frac{p-q}{p-2}) A S_p^{-\frac{q}{2}} \right)^{\frac{1}{2-q}} +o(1) \quad \text{ as } k \to \infty.
\end{align*}

Following the last part of the proof of Lemma \ref{Lemma - Nehari0 is an emptyset }, we get that $J_{\lambda,p}(u_k) = o(1)$ as $k \to \infty.$ On the other hand, we use (\ref{Hardy-Sobolev term >=C1>0})  and  the fact that $\lambda \in (0, \Lambda)$ to get    

\begin{equation}
\begin{aligned}
J_{\lambda,p}(u_k) &\ge C(p,q) \left(       \frac{\|u_k\|^{2(p-1)}}{\int_{\Omega} g(x) \frac{|u_k|^{p}}{|x|^{s}} dx} \right)^{\frac{1}{p-2}} - \lambda A \left(\int_{\Omega}  \frac{|u_k|^{p}}{|x|^{s}} dx \right)^{\frac{q}{p}}\\
& \ge  \left(\int_{\Omega}  \frac{|u_k|^{p}}{|x|^{s}} dx \right)^{\frac{q}{p}}  \left[ C(p,q) S_p^{\frac{p-1}{p-2}+ \frac{q-1}{2-q}} \left( \frac{\lambda A (p-q)}{p-2}\right)^{\frac{1-q}{2-q}}    \|g\|_\infty^{\frac{1}{2-p}}    - \lambda A\right].\\
&>0,
\end{aligned}
\end{equation}
which contradicts $J_{\lambda,p}(u_k) =o(1)$ as $k \to \infty.$ Therefore, (\ref{The denominator of sigma's bound >0}) holds, and there exists a constant $b>0$ such that 
$$ \langle I_{\lambda,p}'(u_k), \frac{u}{\|u\|}\rangle  \le \frac{b}{k}.$$
This  implies (\ref{claim || I' || go to 0}), and completes the proof.\\

2. The proof goes exactly as the first part using Lemma \ref{Lemma - existence of diffrentiable function sigma on Ne^-}.

\end{proof}

\section{Proof of Theorem \ref{Theorem - Sub-Critical} }

In this section, we use the results in section 3 to prove the existence of a positive solution on $\Ne^+,$ as well as on $\Ne^+.$ This coupled with the fact that $\Ne^- \cap \Ne^+ = \emptyset $ yield Theorem \ref{Theorem - Sub-Critical}.

\begin{theorem}\label{Theorem - existence of the solution in N+}
Let $\Lambda= \Lambda(p) := \min\{ \Lambda_1, \Lambda_2, \Lambda_3    \}.$ Then,  for any $\lambda \in (0,\Lambda),$  there exists a minimizer $u_+   \in \Ne^+$ for the functional $I_{\lambda,p}$ which verifies
\begin{enumerate}
\item $I_{\lambda,p} (u_+) = \M = \M^+.$
\item $u_+ $ is positive solution of (\ref{Main problem}).
\item $I_{\lambda,p}(u_+) \to 0 \quad \text{ as } \lambda \to 0.$ 
\end{enumerate}

\end{theorem}

\begin{proof}
Let $(u_k)_{k \in \N} \subset \Ne$ be a minimizing sequence for $I_{\lambda,p}$ such that $I_{\lambda,p}(u) = \M +o(1)$ and $I_{\lambda,p}'(u) = o(1)\text{ in }  (\X)',$ given in the first part of  Proposition \ref{Proposition - existence of minimizing sequence, using Ekeland principle}. It then follows from Lemma \ref{Lemma functional I is coercive and bounded below} and the fractional Sobolev embedding that there exists a sub-sequence $(u_k)_{k \in \N}$ - still denote by $u_k$ - and $u_+ \in \X$ such that

\begin{equation}\label{Weakly and strongly convergence of minimizing sequence}
\begin{aligned}
&u_k \rightharpoonup u_+ \text{ weakly in } \X \\
& u_k \to u_+ \text{ strongly  in } L^r(\Omega) \quad \text{ for every }1\le r <2^*_\alpha. 
\end{aligned}
\end{equation}

We first show that $ \int_{\Omega} f(x) |u_+|^q dx  \neq 0.$ Indeed, suppose $ \int_{\Omega} f(x) |u_+|^q dx  = 0.$ Then, by $(\ref{Weakly and strongly convergence of minimizing sequence})_2$, and the fact that $1<q<2<2^*_\alpha,$ we obtain

$$  \int_{\Omega} f(x) |u_k|^q dx   \to  \int_{\Omega} f(x) |u_+|^q dx  = 0  \quad \text{ as } k \to \infty,$$

which means

$$  \int_{\Omega} f(x) |u_k|^q dx=o(1) \quad \text{ as } k \to \infty. $$
Thus,

\begin{align*}
 \|u_k\|^2 &= \lambda  \int_{\Omega} f(x) |u_k|^q dx + \int_{\Omega} g(x) \frac{|u_k|^p}{|x|^{s}} dx \\
& =\int_{\Omega} g(x) \frac{|u_k|^p}{|x|^{s}} dx+o(1) \quad  \text{ as } k \to \infty,
\end{align*}

and

\begin{align*}
I_{\lambda,p}(u_k)&= \frac{1}{2} \| u_k\|^2  -\frac{\lambda}{q} \int_{\Omega}  f(x) |u_k|^q dx -\frac{1}{p}\int_{\Omega} g(x) \frac{|u_k|^{p}}{|x|^{s}} dx\\
&=\left( \frac{1}{2} - \frac{1}{p}\right) \int_{\Omega} g(x) \frac{|u_k|^{p}}{|x|^{s}} dx +o(1) \quad \text{ as } k \to \infty.
\end{align*}

On the other hand,
 
$$I_{\lambda,p}(u) = \M +o(1)<0 \quad \text{ as } k \to \infty.   $$

This leads us to the following contradiction:

\begin{align*}
0 \le \left( \frac{p-2}{2p} \right) \| u_k\|^2 +o(1) = I_{\lambda,p}(u_k) =\M =o(1) < 0.
\end{align*}

Hence,  $$ \int_{\Omega} f(x) |u_+|^q dx  \neq 0.$$ We now prove that $u_k \to u_+$ strongly  in $\X.$

\begin{align*}
\M = \inf\limits_{u \in \X \setminus \{0\}} I_{\lambda,p}(u) &\le I_{\lambda,p}(u_+) = \frac{1}{2} \| u_+\|^2  -\frac{\lambda}{q} \int_{\Omega}  f(x) |u_+|^q dx-\frac{1}{p}\int_{\Omega} g(x) \frac{|u_+|^{p}}{|x|^{s}} dx\\
&=\left( \frac{1}{2} - \frac{1}{p}\right) \int_{\Omega} g(x) \frac{|u_+|^{p}}{|x|^{s}} dx + \left( \frac{1}{p} - \frac{1}{q}\right) \lambda \int_{\Omega}  f(x) |u_+|^q dx \\
&\le \left( \frac{1}{2} - \frac{1}{p}\right) \int_{\Omega} g(x) \frac{|u_k|^{p}}{|x|^{s}} dx + \left( \frac{1}{p} - \frac{1}{q}\right) \lambda \int_{\Omega}  f(x) |u_k|^q dx \\
&= \M.
\end{align*}

This yields $ I_{\lambda,p}(u_+) = \M,$ and  $u_k \to u_+$ strongly  in $\X.$

The next step is to prove that $u_+ \in \Ne^+.$ Assume that $u_+ \in \Ne^-.$ It then follows from Lemma \ref{Lemma - existence of t+ and t-} that there exist $t^-$ and $t^+$ such that $t^- u_+ \in \Ne^-,$ $t^+ u_+ \in N^+$ and $t^+ < t^- =1.$ Following the proof of Lemma \ref{Lemma - existence of t+ and t-}, we have that 
$\frac{d}{dt} I_{\lambda,p}(t^+u) =0  \text{ and }  \frac{d^2}{dt^2} I_{\lambda,p}(t^+u) > 0.$ Thus, there exists a $\tilde{t}$ such that $t^+ < \tilde{t} <t^-=1 $ and  $I_{\lambda,p}(t^+ u_+) < I_{\lambda,p}(\tilde{t} u_+). $ We again use Lemma \ref{Lemma - existence of t+ and t-} to get 

$$I_{\lambda,p}(t^+ u_+) < I_{\lambda,p}(\tilde{t} u_+) \le I_{\lambda,p}(t^- u_+) = I_{\lambda,p}(u_+),$$

which is in contradiction with  $I_{\lambda,p}(u_+)=\M.$ Therefore, $u_+ \in \Ne^+,$ and $I_{\lambda,p} (u_+) = \M = \M^+.$
Since  $I_{\lambda,p} (u_+)=I_{\lambda,p} (|u_+|),$ and $|u_+| \in \Ne^+$ is a solution for (\ref{Main problem}), without loss of generality, we may assume  that $u_+$ is a non-negative solution of (\ref{Main problem}), and strong maximum principle \cite[Proposition 2.2.8]{Silvestre} implies that $u_+ >0 $ in $\Omega.$
 
To complete the proof of  Theorem \ref{Theorem - existence of the solution in N+}, we need to show that $I_{\lambda,p}(u_+) \to 0 \text{ as } \lambda \to 0.$ From Lemma \ref{Lemma functional I is coercive and bounded below}, it follows 

$$ - \lambda \left(\frac{p-q}{pq} \right)  \left(\frac{2-q}{2} \right) \left( A S_p^{-\frac{q}{2}} \right)^{\frac{2}{2-q}} \le I_{\lambda,p}(u_+) < \M < 0.$$ 

Thus, $I_{\lambda,p}(u_+) \to 0 \text{ as } \lambda \to 0.$ 
\end{proof}

\begin{theorem}\label{Theorem - existence of the solution in N-}
Let $\Lambda= \Lambda(p) := \min\{ \Lambda_1, \Lambda_2, \Lambda_3    \}.$ Then, for any $\lambda \in (0,\Lambda),$ the functional $I_{\lambda,p}$ has a minimizer $u_- \in \Ne^-$  which verifies
\begin{enumerate}
\item $I_{\lambda,p} (u_-) = \M^-.$
\item $u_- $ is a positive solution of (\ref{Main problem}).
\end{enumerate}

\end{theorem}

\begin{proof}

Let $(u_k)_{k \in \N} \subset \Ne^-$ be a minimizing sequence for $I_{\lambda,p}$ such that $I_{\lambda,p}(u) = \M^- +o(1)$ and $I_{\lambda,p}'(u) = o(1)\text{ in }  (\X)',$ given in the second part of  Proposition \ref{Proposition - existence of minimizing sequence, using Ekeland principle}. It then follows from Lemma \ref{Lemma functional I is coercive and bounded below} and the fractional Sobolev embedding that there exists a sub-sequence $(u_k)_{k \in \N}$ - still denote by $u_k$ - and $u_- \in \X$ such that

\begin{equation}\label{Weakly and strongly convergence of minimizing sequence Ne -}
\begin{aligned}
&u_k \rightharpoonup u_- \text{ weakly in } \X \\
& u_k \to u_- \text{ strongly  in } L^r(\Omega) \quad \text{ for every }1\le r <2^*_\alpha. 
\end{aligned}
\end{equation}

We prove that $u_k \to u_-$ in $\X.$ Indeed, if not, then we have $\|u_-\| < \liminf\limits_{k \to \infty} \|u_k\|.$ Therefore,

\begin{align*}
\langle I_{\lambda,p}'(u_-), u_- \rangle &=\| u_-\|^2  - \lambda \int_{\Omega}  f(x) |u_-|^q dx -\int_{\Omega} g(x) \frac{|u_-|^{p}}{|x|^{s}} dx\\
& < \liminf\limits_{k \to \infty}
 \left(   \| u_k\|^2  - \lambda \int_{\Omega}  f(x) |u_k|^q dx -\int_{\Omega} g(x) \frac{|u_k|^{p}}{|x|^{s}} dx  \right )\\
& =0,
 \end{align*}

which contradicts $u_- \in \Ne^-.$ This implies that $u_k \to u_-$ in $\X,$ and therefore $I_{\lambda,p} (u_-) = \M^-.$ Since $I_{\lambda,p} (u_-) = I_{\lambda,p} (|u_-|),$ and  $|u_-| \in \Ne^-$ is a solution for (\ref{Main problem}), without loss of generality, we may assume  that $u_-$ is a non-negative solution of (\ref{Main problem}), and the maximum principle \cite[Proposition 2.2.8]{Silvestre} implies that $u_- >0$ in $\Omega.$

\end{proof}

\begin{proof}[ Proof of Theorem \ref{Theorem - Sub-Critical}]

It follows from Theorems \ref{Theorem - existence of the solution in N+} and \ref{Theorem - existence of the solution in N-}
that there exist two positive solutions $u_+$ and $u_-$ such that $u_+ \in \Ne^+$  and $u_- \in \Ne^-.$ In addition, by Lemma \ref{Lemma - Nehari0 is an emptyset }, $\Ne^+ \cap \Ne^- = \emptyset.$ Thus,  $u_+$ and $u_-$ are two distinct positive solutions for (\ref{Main problem}).

\end{proof}

\section{Proof of Theorem \ref{Theorem - Critical}}

Throughout this section, we shall assume that
\begin{align}\label{Condition-Critical case}
p =2_{\alpha}^*(s) \quad  \text{ and } \quad  g(x) \equiv 1.
\end{align}

We also use the following notations for simplicity: 

$$\Lambda^*:= \Lambda(2^*_\alpha(s)) \quad  \text{ and } \quad   \I(u) := I_{\lambda, 2^*_\alpha(s)}(u).$$

We point out that all results (i.e., Lemmas, Propositions and Theorems) stated in the previous sections hold under condition $(\ref{Condition-Critical case}).$

The first step to prove Theorem \ref{Theorem - Critical} is to study the existence and asymptotic behavior of the weak solutions to the following borderline problem associated with the fractional Hardy-Schr\"odinger operator $ ({-}{ \Delta})^{\frac{\alpha}{2}}-  \frac{\gamma}{|x|^{\alpha}}$ on $\R^n:$
 \begin{equation}\label{Limiting problem}
\left\{\begin{array}{rl}
({-}{ \Delta})^{\frac{\alpha}{2}}u- \gamma \frac{u}{|x|^{\alpha}}= {\frac{u^{2_{\alpha}^*(s)-1}}{|x|^s}} & \text{in }  {\R^n}\\
 u>0 \,\,\,\,\,\,\,\,\,\,\,\,\,\,\,\,\,  & \text{in }  \mathbb{R}^n,
\end{array}\right.
\end{equation}

where $0<\alpha<2$, $ 0 \le s < \alpha$,  $ {2_{\alpha}^*(s)}={\frac{2(n-s)}{n-{\alpha}}}$,   
$ 0 \le \gamma < \gamma_H(\alpha)=2^\alpha \frac{\Gamma^2(\frac{n+\alpha}{4})}{\Gamma^2(\frac{n-\alpha}{4})}.$

The existence of the weak solutions to (\ref{Limiting problem}) was proved in \cite{Ghoussoub-Shaya}. Recently, the author and et al. in \cite{Ghoussoub-Robert-Zhao-Shaya} have proved the following results regarding the asymptotic behavior of such solutions which play a crucial role in this section:

\begin{theorem}[Theorem 1.2  in \cite{Ghoussoub-Robert-Zhao-Shaya}]\label{Theorem-Upper and lower bound for the solution on Rn}
Assume $0 \le s<\alpha<2,$ $n> \alpha$ and $\displaystyle 0\leq\gamma<\gamma_H(\alpha)$. Then, any positive solution $\displaystyle u\in H_0^{\frac{\alpha}{2}}(\R^n)$ of (\ref{Limiting problem}) satisfies $u\in C^1(\R^n\setminus\{0\})$ and  
\begin{equation}\label{Asymptotic behaviour at zero and infty}
\lim_{x\to 0}|x|^{\Bm}u(x)=\lambda_0\hbox{ and }\lim_{|x|\to \infty}|x|^{\Bp}u(x)=\lambda_\infty,
\end{equation}
where $\lambda_0,\lambda_\infty>0$ and $\Bm$ (resp., $\Bp$) is the unique solution in 
$\left(0,\frac{n-\alpha}{2}\right)$ (resp., in $\left(\frac{n-\alpha}{2},n-\alpha\right))$ of the equation 
$$\Psi_{n,\alpha}(t):=2^\alpha\frac{\Gamma\left(\frac{t+\alpha}{2}\right)\Gamma\left(\frac{n-t}{2}\right)}{\Gamma\left(\frac{n-t-\alpha}{2}\right)\Gamma\left(\frac{t}{2}\right)}=\gamma,$$

with $\beta_-(0) =0,$ and  $\beta_+(0) =n-\alpha.$


\end{theorem}

We refer the readers to Section 2 in \cite{Ghoussoub-Robert-Zhao-Shaya} for the definition and properties of $\Bp$ and $\Bm$ in detail

Let $u^*(x)$ be a positive weak solution of (\ref{Main problem-critical}). For any $\epsilon >0,$ we define $u_\epsilon(x)= \epsilon^{\frac{\alpha-n}{2}} u^*(\frac{x}{\epsilon} ) $ in $\R^n.$ It is easy to show that $u_\epsilon(x)$ is also a solution of (\ref{Main problem-critical}). From the assumption on $f$, we know that $f$ is a continuous function, and also  $f^+(x) =\max \{f(x),0\} \not\equiv 0.$  Let $\Sigma:= \left\{x \in \Omega : f(x) >0   \right\}$ be an open set of positive measure. Now we need to define appropriate cut-off function. Let $\eta \in C^\infty_0(\Sigma)$ be a positive cut-off function satisfying $0\le \eta \le 1$ in $\Sigma.$ 
 In addition, we choose $\rho >0  $ small enough such that $B^c_{2\rho} \subset \overline{\Sigma},$  $\eta \equiv 1$ in $B_\rho,$ and $\eta \equiv 0$ in  $B^c_{2\rho}.$ One can check that $\eta  u_\epsilon(x)$ is in $\X.$  For any $\epsilon>0,$ we define 

\begin{equation}\label{Definition of U_epsilon}
U_\epsilon(x)= \eta(x)  u_\epsilon(x) \quad \text{ for } x \in \R^n.
\end{equation}

The following lemma is a direct consequence of the computations in Section 6.1 in \cite{Ghoussoub-Robert-Zhao-Shaya}: 
\begin{lemma}\label{Lemma - estimate for eta u_epsilon for norm and H-S term}
Assume that $U_\epsilon$ defined as (\ref{Definition of U_epsilon}), and that $u_1$ be a positive solution of (\ref{Main problem-critical}). Then, for every $\epsilon >0$ small enough, we have\\

(i) $\| U_\epsilon \|^2\le \| u_\epsilon \|^2 +O(\epsilon^{\Bp - \Bm}). $\\

(ii) $ \int_{\Omega}\frac{ |U_\epsilon|^{2_{\alpha}^*(s)}}{|x|^s}dx = \int_{\Omega}\frac{ |u_\epsilon|^{2_{\alpha}^*(s)}}{|x|^s}dx  + o(\epsilon^{\Bp - \Bm}).$\\



\end{lemma}











We need the following two lemmas in order  to prove Theorem \ref{Theorem - Critical}.

\begin{lemma}\label{Lemma-Estimates for all terms in the functional }
Assume that $U_\epsilon$ defined as  (\ref{Definition of U_epsilon}), and that $u_1$ be the local minimum in Theorem \ref{Theorem - existence of the solution in N+ - critical}. Then, for every $\epsilon >0$ small enough, we have

\begin{equation}\label{estimate H-S term at u_1+ t eta u_epsilon}
\begin{aligned}
\int_{\Omega}  \frac{|u_1 + t U_\epsilon|^{2_{\alpha}^*(s)}}{|x|^{s}} dx & =  \int_{\Omega}  \frac{|u_1|^{2_{\alpha}^*(s)}}{|x|^{s}} dx + \int_{\Omega}  \frac{|t U_\epsilon|^{2_{\alpha}^*(s)}}{|x|^{s}} dx +   2_{\alpha}^*(s) t\int_{\Omega}  \frac{|u_1|^{2_{\alpha}^*(s)-2}}{|x|^{s}} U_\epsilon u_1 dx\\
& \quad +2_{\alpha}^*(s) t^{2_{\alpha}^*(s)-1} \int_{\Omega}  \frac{| U_\epsilon|^{2_{\alpha}^*(s)-2}}{|x|^{s}} U_\epsilon u_1 dx +o(\epsilon^{\frac{\Bp-\Bm}{2}}).
\end{aligned}
\end{equation}

\end{lemma}
\begin{proof}

The proof goes exactly as (17)  in \cite[Theorem 1]{Brezis-Nirenberg-nonzero data} with only minor modifications. We omit it here.







\end{proof}

\subsection{The existence of a minimizer on $\Ne^+$}

In the following theorem, we prove the existence of a positive solution of (\ref{Main problem-critical}) on $\Ne^+.$
\begin{theorem}\label{Theorem - existence of the solution in N+ - critical}
 For any $\lambda \in (0,\Lambda^*),$ there exists a minimizer  $u_1 \in \Ne^+$ for the functional $\I$ which verifies
\begin{enumerate}
\item $\I (u_1) = \M = \M^+.$
\item $u_1 $ is positive solution of (\ref{Main problem-critical}).
\item $\I(u_1) \to 0 \quad \text{ as } \lambda \to 0.$ 
\end{enumerate}

\end{theorem}

\begin{proof}
The proof is a straightforward consequence of Theorem \ref{Theorem - existence of the solution in N+} with $ p = 2_\alpha^*(s).$
\end{proof}

\subsection{The existence of a minimizer on $\Ne^-$}


In obtaining the existence result on $\Ne^-,$ it is crucial 
to have the (P.S) conditions for all level $\sigma < \M + \frac{\alpha - s}{2(n-s)} S_p^{\frac{n-s}{\alpha-s}},$ which will be shown in the next two lemmas.

\begin{lemma}
Let $u_1$ be the local minimum in Theorem \ref{Theorem - existence of the solution in N+ - critical} . Then, for $\epsilon>0$ small enough, we have 

$$ \sup\limits_{t \ge 0}\I(u_1 + t U_\epsilon ) < \M + \frac{\alpha - s}{2(n-s)} S_p^{\frac{n-s}{\alpha-s}}.$$
\end{lemma}

\begin{proof}
We first note that

$$\I(u_1 + t U_\epsilon ) = \frac{1}{2} \| u_1 + t U_\epsilon\|^2  -\frac{\lambda}{q} \int_{\Omega}  f(x) |u_1 + t U_\epsilon|^q dx -\frac{1}{2_{\alpha}^*(s)}\int_{\Omega}  \frac{|u_1 + t U_\epsilon|^{2_{\alpha}^*(s)}}{|x|^{s}} dx.$$

On the other hand, simple computations yield 

\begin{align*}
\| u_1 + t U_\epsilon\|^2 = \|u_1\|^2 + t^2 \|U_\epsilon \|^2 + 2 t \langle u_1 , U_\epsilon \rangle_{\X} - 2 \gamma \int_{\Omega} \frac{u_1 U_\epsilon}{|x|^\alpha} dx.
\end{align*}

Thus,

\begin{equation}\label{the functiona I-lambda at (u_1+t eta u_epsilon)}
\begin{aligned}
\I(u_1 + t U_\epsilon ) &= \frac{1}{2} \|u_1\|^2 + \frac{t^2}{2} \|U_\epsilon\|^2 + t \langle u_1 , U_\epsilon \rangle_{\X} - \gamma \int_{\Omega} \frac{u_1 U_\epsilon}{|x|^\alpha} dx\\
&   -\frac{\lambda}{q} \int_{\Omega}  f(x) |u_1 + t U_\epsilon|^q dx -\frac{1}{2_{\alpha}^*(s)}\int_{\Omega}  \frac{|u_1 + t U_\epsilon|^{2_{\alpha}^*(s)}}{|x|^{s}} dx.
\end{aligned}
\end{equation}
Now we deal with each terms separately:

Regarding the first term, since $u_1$ is a minimizer for $\I,$ we have 

$$\frac{1}{2} \|u_1\|^2 = \I(u_1) + \frac{\lambda}{q} \int_{\Omega} f(x) |u_1|^q dx + \frac{1}{2_{\alpha}^*(s)}\int_{\Omega}  \frac{|u_1|^{2_{\alpha}^*(s)}}{|x|^{s}} dx.$$

For the third one, we substitute test function $\eta u_1$ into  $\I'(u)=0 \text{  in } \X$ to get

$$t \langle u_1 , U_\epsilon \rangle_{\X} - \gamma \int_{\Omega} \frac{u_1 U_\epsilon}{|x|^\alpha} dx =t \lambda  \int_{\Omega} f(x) |u_1|^{q-1}  U_\epsilon dx + t \int_{\Omega}  \frac{|u_1|^{2_{\alpha}^*(s)-1}}{|x|^{s}} U_\epsilon dx. $$

Plugging the last two inequalities and (\ref{estimate H-S term at u_1+ t eta u_epsilon}) into (\ref{the functiona I-lambda at (u_1+t eta u_epsilon)}), we obtain

\begin{equation}
\begin{aligned}
\I(u_1 + t U_\epsilon ) &= \I(u_1) -  \frac{\lambda}{q} \int_{\Sigma} \left(  f(x) |u_1 + t U_\epsilon|^q dx  - f(x) |u_1|^q dx - tq   f(x) |u_1|^{q-1}  U_\epsilon \right) dx \\
& + \frac{t^2}{2} \|U_\epsilon \|^2   - \frac{t^{2_{\alpha}^*(s)}}{2_{\alpha}^*(s)}\int_{\Omega}  \frac{| U_\epsilon|^{2_{\alpha}^*(s)}}{|x|^{s}} dx - t^{2_{\alpha}^*(s)-1} \int_{\Omega}  \frac{|U_\epsilon|^{2_{\alpha}^*(s)-1}}{|x|^{s}}  u_1 dx +o(\epsilon^{\frac{\Bp-\Bm}{2}}).
\end{aligned}
\end{equation}

We also have 

\begin{align*}
&\int_{\Sigma} \left(  f(x) |u_1 + t U_\epsilon|^q dx  - f(x) |u_1|^q dx - tq   f(x) |u_1|^{q-1}  U_\epsilon \right) dx\\
& \quad = q \int_{\Sigma} f(x) \left(     \int^{t U_\epsilon}_0 |u_1+\tau |^{q-1} - |u_1|^{q-1} d \tau \right) dx \\
&  \quad \ge q \int_{\Sigma} f^+(x) \left(     \int^{t U_\epsilon}_0 |u_1+\tau |^{q-1} - |u_1|^{q-1} d \tau \right) dx\\
& \quad \ge 0.
\end{align*}

In addition, we know that $u_1$ is a positive solution of (\ref{Main problem-critical}).   Following the iterative scheme used to prove Proposition 3.3 in \cite{Ghoussoub-Robert-Zhao-Shaya}, one can show that

$$u_1(x) \le C |x|^{-\Bm} \quad  \text{ for  all} \ x \in \Omega.  $$



Thus,

\begin{align*}
\int_{\Omega} \frac{|U_\epsilon|^{2_{\alpha}^*(s)-1}}{|x|^{s}} u_1 dx 
&\le C \int_{\Omega} \frac{|U_\epsilon|^{2_{\alpha}^*(s)-1}}{|x|^{s}} |x|^{-\Bm} dx\\
& = C \int_{B_{\delta}} \frac{|U_\epsilon|^{2_{\alpha}^*(s)-1}}{|x|^{s}} |x|^{-\Bm} dx + C \int_{\Omega \setminus B_{\delta}} \frac{|U_\epsilon|^{2_{\alpha}^*(s)-1}}{|x|^{s}} |x|^{-\Bm} dx\\
&= C \epsilon^{n+ \frac{\alpha-n}{2}l-s -\Bm} \int_{B_{\epsilon^{-1} \delta}}  \frac{|u^*|^{2_{\alpha}^*(s)-1}}{|x|^{s}} |x|^{-\Bm} dx +o(\epsilon^{\frac{\Bp-\Bm}{2}})\\
& =  C \epsilon^{n+ \frac{\alpha-n}{2}l-s -\Bm} \int_{\R^n}  \frac{|u^*|^{2_{\alpha}^*(s)-1}}{|x|^{s}} |x|^{-\Bm}dx +o(\epsilon^{\frac{\Bp-\Bm}{2}}) \\
& = C \epsilon^{\frac{\Bp-\Bm}{2}} \int_{\R^n}  \frac{|u^*|^{2_{\alpha}^*(s)-1}}{|x|^{s}} |x|^{-\Bm}dx +o(\epsilon^{\frac{\Bp-\Bm}{2}}) \\
&= K \epsilon^{\frac{\Bp-\Bm}{2}} +o(\epsilon^{\frac{\Bp-\Bm}{2}}) \quad \text{ for some } K>0, \quad \text{ as }\epsilon \to 0.
\end{align*}



Note that one can use the asymptotic (\ref{Asymptotic behaviour at zero and infty}) in Theorem \ref{Theorem-Upper and lower bound for the solution on Rn} to show that the last integral is finite.

Therefore, there exist $c>0$ such that

\begin{equation*}
\begin{aligned}
\I(u_1 + t U_\epsilon ) &\le  \I(u_1) + \frac{t^2}{2} \| u_\epsilon \|^2   - \frac{t^{2_{\alpha}^*(s)}}{2_{\alpha}^*(s)}\int_{\Omega}  \frac{|  u_\epsilon|^{2_{\alpha}^*(s)}}{|x|^{s}} dx  -  c \ \epsilon^{\frac{\Bp-\Bm}{2}} + o(\epsilon^{\frac{\Bp-\Bm}{2}}).
\end{aligned}
\end{equation*}

We now define

$$m(t) = \frac{t^2}{2} \| u_\epsilon \|^2  - \frac{t^{2_{\alpha}^*(s)}}{2_{\alpha}^*(s)}\int_{\Omega}  \frac{|  u_\epsilon|^{2_{\alpha}^*(s)}}{|x|^{s}} dx \quad \text{ for } t>0. $$

By straightforward computations, we get that $m$ attained its maximum at $\tilde{t} = \left( \frac{ \| u_\epsilon \|^2}{\int_\Omega \frac{|  u_\epsilon|^{2_{\alpha}^*(s)}}{|x|^{s}} dx }\right)^{\frac{1}{2_{\alpha}^*(s) -2}}, $ $\lim\limits_{t \to \infty} m(t) =-\infty,$ and also

$$m(\tilde{t}) =\left( \frac{1}{2} - \frac{1}{2_{\alpha}^*(s)}\right)  \|u_\epsilon\|^{\frac{2 2_{\alpha}^*(s)}{2_{\alpha}^*(s) -2}} \left(\int_\Omega \frac{|  u_\epsilon|^{2_{\alpha}^*(s)}}{|x|^{s}} dx    \right)^{-\frac{2}{2_{\alpha}^*(s) -2}}.   $$

Thus, for all $t>0,$

$$m(t) \le \left( \frac{1}{2} - \frac{1}{2_{\alpha}^*(s)}\right)  \|u_\epsilon\|^{\frac{2 2_{\alpha}^*(s)}{2_{\alpha}^*(s) -2}} \left(\int_\Omega \frac{|  u_\epsilon|^{2_{\alpha}^*(s)}}{|x|^{s}} dx    \right)^{-\frac{2}{2_{\alpha}^*(s) -2}}.   $$

On the other hand, since $u_\epsilon$ is an extremal for  (\ref{fractional H-S-M inequality}), we have

$$\|u_\epsilon\|^2 = S_p  \left(\int_\Omega \frac{|  u_\epsilon|^{2_{\alpha}^*(s)}}{|x|^{s}} dx    \right)^{\frac{2}{2_{\alpha}^*(s)}}.$$
 
Hence,

 $$ S_p^{\frac{2_{\alpha}^*(s)}{2_{\alpha}^*(s) -2}}=\|u_\epsilon\|^{\frac{2 2_{\alpha}^*(s)}{2_{\alpha}^*(s) -2}} \left(\int_\Omega \frac{|  u_\epsilon|^{2_{\alpha}^*(s)}}{|x|^{s}} dx    \right)^{-\frac{2}{2_{\alpha}^*(s) -2}}.$$

Noting that $\frac{1}{2} - \frac{1}{2_{\alpha}^*(s)} =\frac{\alpha -s}{2(n-s)} $ and $\frac{2_{\alpha}^*(s)}{2_{\alpha}^*(s) -2} =\frac{n-s}{\alpha-s} ,$ we get 

$$m(t) \le \frac{\alpha -s}{2(n-s)} S_p^{\frac{n-s}{\alpha-s}} \quad \text{ for all } t >0 .$$  
 
Therefore,

\begin{equation*}
\begin{aligned}
\I(u_1 + t U_\epsilon )& \le \I(u_1)+ \frac{\alpha -s}{2(n-s)} S_p^{\frac{n-s}{\alpha-s}} - c \ \epsilon^{\frac{\Bp-\Bm}{2}} + o(\epsilon^{\frac{\Bp-\Bm}{2}}).\\
& <  \M +\frac{\alpha -s}{2(n-s)} S_p^{\frac{n-s}{\alpha-s}} \quad  \text{ for all } t>0.
\end{aligned}
\end{equation*} 
 
\end{proof}

\begin{lemma}\label{Lemma - any sequence which satisfies two conditions has strongly convergence in X}
Suppose that a sequence $(u_k)_{k \in \N}$ satisfies  the following:

\begin{enumerate}
\item $\I(u_k) = \sigma +o(1) \quad \text{with } \sigma < \M +\frac{\alpha -s}{2(n-s)} S_p^{\frac{n-s}{\alpha-s}} $
\item $\I'(u_k) =o(1) \quad \text{in } (\X)'$
\end{enumerate}
Then, there exists a sub-sequence of $(u_k)_{k \in \N}$ which is strongly convergence in $\X.$
\end{lemma} 

\begin{proof}

It follows from Lemma \ref{Lemma functional I is coercive and bounded below}  that  $(u_k)_{k \in \N}$ is bounded in $\X.$ Then, there exists a sub-sequence - still donote by $u_k$ - and $u$ such that 

\begin{equation}\label{Weakly and strongly convergence of minimizing sequence- critical case}
\begin{aligned}
&u_k \rightharpoonup u \text{ weakly in } \X \\
& u_k \to u \text{ strongly  in } L^r(\Omega) \quad \text{ for every }1\le r <2^*_\alpha. 
\end{aligned}
\end{equation}

Consequently from the second assumption, we obtain 

$$ \langle  I'_\lambda (u) , w \rangle = 0  \quad \forall w \in \X.$$

Then, $u$ is a solution in $\X$ for (\ref{Main problem-critical}) with $I_{\lambda}(u) \ge \M. $

We first prove that $u \not\equiv 0.$ Indeed, suppose $u \equiv 0.$ Then, by $(\ref{Weakly and strongly convergence of minimizing sequence- critical case})_2$, and the fact that $1<q<2<2^*_\alpha,$ we obtain

$$  \int_{\Omega} f(x) |u_k|^q dx   \to  \int_{\Omega} f(x) |u|^q dx  = 0, $$ 

which implies

$$ \int_{\Omega} f(x) |u_k|^q dx=o(1) \quad \text{ as } k \to \infty. $$

Thus, the second assumption yields 

\begin{equation} \label{Critical case - ||u|| = int H-S term + o(1)}
\begin{aligned}
 \|u_k\|^2 &= \lambda  \int_{\Omega} f(x) |u_k|^q dx + \int_{\Omega}  \frac{|u_k|^{2^*_\alpha(s)}}{|x|^{s}} dx \\
& =\int_{\Omega}  \frac{|u_k|^{2^*_\alpha(s)}}{|x|^{s}} dx+o(1)\quad \text{ as } k \to \infty, 
\end{aligned}
\end{equation}

and the first assumption then implies that

\begin{align*}
I_{\lambda,p}(u_k)&= \frac{1}{2} \| u_k\|^2  -\frac{\lambda}{q} \int_{\Omega}  f(x) |u_k|^q dx -\frac{1}{2^*_\alpha(s)}\int_{\Omega}  \frac{|u_k|^{{2^*_\alpha(s)}}}{|x|^{s}} dx\\
&=\left( \frac{1}{2} - \frac{1}{2^*_\alpha(s)}\right) \int_{\Omega} \frac{|u_k|^{2^*_\alpha(s)}}{|x|^{s}} dx +o(1)\\
& = \frac{\alpha -s}{2(n-s)}  \int_{\Omega} \frac{|u_k|^{2^*_\alpha(s)}}{|x|^{s}} dx +o(1)\\
& = \sigma +o(1) \qquad  \text{ as }  k \to \infty.
\end{align*}

Since $\sigma < \frac{\alpha -s}{2(n-s)} S_p^{\frac{n-s}{\alpha-s}},$ we get that  

$$ \int_{\Omega} \frac{|u_k|^{2^*_\alpha(s)}}{|x|^{s}} dx < S_p^{\frac{n-s}{\alpha-s}}+o(1) \quad \text{ as } k \to \infty.$$

On the other hand, it follows from  (\ref{fractional H-S-M inequality}) and (\ref{Critical case - ||u|| = int H-S term + o(1)}) that

$$ \int_{\Omega} \frac{|u_k|^{2^*_\alpha(s)}}{|x|^{s}} dx \ge S_p^{\frac{n-s}{\alpha-s}}+o(1) \quad \text{ as } k \to \infty.$$

This gives us a contradiction which implies that $u$ can not be identically zero, and thus $u \not \equiv 0$ with $\I(u) \ge \M.$

 Let now $v_k = u_k - u,$ for all $k \in \N.$ We may verify as Brézis-Lieb lemma in \cite{Brezis-Lieb} that (see also \cite[Lemma 4.2]{Ghoussoub-Yuan})

$$\int_{\Omega} \frac{|u_k|^{2^*_\alpha(s)}}{|x|^{s}} dx= \int_{\Omega} \frac{|u|^{2^*_\alpha(s)}}{|x|^{s}} dx+ \int_{\Omega} \frac{|v_k|^{2^*_\alpha(s)}}{|x|^{s}} dx +o(1) \quad \text { as } k \to \infty.$$

Hence, by weakly convergence $v_k \rightharpoonup 0$ in $\X$, we can conclude that

\begin{align*}
 \M +\frac{\alpha -s}{2(n-s)} S_p^{\frac{n-s}{\alpha-s}} & > \I(u_2 + v_k)\\
 & = \I(u_2) + \frac{1}{2} \|v_k\|^2 - \frac{1}{2^*_\alpha(s)}\int_{\Omega} \frac{|v_k|^{2^*_\alpha(s)}}{|x|^{s}} dx +o(1)\\
 & \ge \M + \frac{1}{2} \|v_k\|^2 - \frac{1}{2^*_\alpha(s)}\int_{\Omega} \frac{|v_k|^{2^*_\alpha(s)}}{|x|^{s}} dx +o(1) \quad \text{ as } k \to \infty.
\end{align*}

 Then, 

\begin{equation}\label{Norm ||v_k ||_X - H-S term is bounded by S_p}
 \frac{1}{2} \|v_k\|^2 - \frac{1}{2^*_\alpha(s)} \int_{\Omega} \frac{|v_k|^{2^*_\alpha(s)}}{|x|^{s}} dx < \frac{\alpha -s}{2(n-s)} S_p^{\frac{n-s}{\alpha-s}}  +o(1) \quad \text{ as } k \to \infty.
\end{equation}

On the other hand, from the second assumption, we know that $(u_k)_{k \in \N}$ is uniformly bounded and $u$ is  solution of (\ref{Main problem-critical}). So,

\begin{align*}
o(1) &= \I' (u_k,u_k)\\
& = \frac{1}{2} \| u_k\|^2  -\lambda \int_{\Omega}  f(x) |u_k|^q dx -\int_{\Omega}  \frac{|u_k|^{{2^*_\alpha(s)}}}{|x|^{s}} dx\\
& = \I'(u) + \|v_k\|^2 - \int_{\Omega}  \frac{|v_k|^{{2^*_\alpha(s)}}}{|x|^{s}} dx +o(1)  \quad \text{ as } k \to \infty. 
\end{align*} 

Since $\I'(u) =0,$ we have

\begin{equation}\label{Norm ||v_k|| - H-S term is equal to o(1).}
\|v_k\|^2 - \int_{\Omega}  \frac{|v_k|^{{2^*_\alpha(s)}}}{|x|^{s}} dx = o(1)  \quad \text{ as } k \to \infty. 
\end{equation}

Now we prove that if (\ref{Norm ||v_k ||_X - H-S term is bounded by S_p})  and (\ref{Norm ||v_k|| - H-S term is equal to o(1).}) hold, then $(v_k)_{k \in \N}$ admits a sub-sequence which converges strongly to zero. Indeed, if not, there exits a constant $c>0$ such that $\|v_k\|^2_{\X} \ge c >0,$ for all $k \in \N.$ Combining (\ref{Norm ||v_k ||_X - H-S term is bounded by S_p}) and (\ref{Norm ||v_k|| - H-S term is equal to o(1).}) leads us to the following contradiction:

 \begin{align*}
 \frac{\alpha -s}{2(n-s)} S_p^{\frac{n-s}{\alpha-s}}&  \le \frac{\alpha -s}{2(n-s)} \|v_k\|^2 + o(1)\\
 & = \frac{1}{2} \|v_k\|^2_{\X}  - \frac{1}{2_\alpha^*(s)} \|v_k\|^2 +o(1)\\
 & < \frac{\alpha -s}{2(n-s)} S_p^{\frac{n-s}{\alpha-s}} +o(1) \quad \text{ as } k \to \infty.
 \end{align*}

Therefore, up to a sub-sequence, $v_k \to 0 $ strongly in $\X.$ This implies that $u_k \to u$ strongly in $\X.$

\end{proof}

We are now ready to prove the existence results on $\Ne^+.$

\begin{proposition}\label{Proposition - existence of the solution in N- - critical}
For any $\lambda \in (0,\Lambda^*),$ there exists a minimizer  $u_2 \in \Ne^-$ for the functional $\I$ which verifies

\begin{enumerate}
\item $\I(u_2) = \M^- < \M + \frac{\alpha -s}{2(n-s)} S_p^{\frac{n-s}{\alpha-s}} .$
\item $ u_2 $ is a nontrivial non-negative solution of (\ref{Main problem-critical}).
\end{enumerate}
\end{proposition}

\begin{proof}

We first show that

\begin{equation}\label{M- < M + alpha-s/2(n-s) (s_p)^ n-s/alpha-s }
 \M^- < \M + \frac{\alpha -s}{2(n-s)} S_p^{\frac{n-s}{\alpha-s}} .
\end{equation} 

Let

$$W_1: = \left\{u \in \X \setminus \{0\}: t^-(\frac{u}{\|u\|}) >        \|u\| \right\}\cup\{0\},$$
and

$$W_2: = \left\{u \in \X \setminus \{0\}: t^-(\frac{u}{\|u\|}) <        \|u\| \right\}.$$

Thus, $\Ne^-$ disconnects  $\X$ in two connected components $W_1$ and $W_2,$ and $\X \setminus \Ne^- = W_1 \cup W_2.$  By Lemma \ref{Lemma - existence of t+ and t-}, for any $u \in \Ne^+,$ there exists a unique $t^-(\frac{u}{\|u\|})>0$ such that $1 < t_{max}<t^-(u).$ Since $t^-(u) = \frac{1}{\|u\|} t^-(\frac{u}{\|u\|}).$ Then, $t^-(\frac{u}{\|u\|}) >        \|u\|, $ and $\Ne^+ \subset W_1.$ In particular, $u_1 \in W_1.$ 

Next step is to show that there exists $n_0>0$ such that $u_1 + n_0 U_\epsilon \in W_2.$  To prove this, we first note that there exists $C>0$ such that

\begin{equation} \label{0< t-(u_1+ n_0 eta u_epsilon/ ||u_1+ n_0 eta u_epsilon ||) <C}
0< t^- (\frac{u_1 + n_0 U_\epsilon}{\|u_1 + n_0 U_\epsilon\|}) <C \quad  \text{ for all } n_0 >0.
\end{equation}

Indeed, if not, there exists a sub-sequence $(n_k)_{k \in \N}$ such that    

$$n_k \to \infty \quad \text{ and } \quad  t^- (\frac{u_1 + n_k U_\epsilon}{\|u_1 + n_k U_\epsilon\|}) \to 0 \quad \text{ as } k \to \infty. $$

For all $k \in \N,$ let $v_k = \frac{u_1 + n_k U_\epsilon}{\|u_1 + n_k U_\epsilon\|}.$ So,  Lemma \ref{Lemma - existence of t+ and t-} implies that $t^- (v_k) v_k \in \Ne^- \subset \Ne$  for all $k \in \Ne.$ 

Then, a straightforward computation and the Lebesgue
dominated convergence theorem yield  

\begin{align*}
\int_{\Omega}  \frac{|v_k|^{{2^*_\alpha(s)}}}{|x|^{s}} dx &= \frac{1}{\|u_1 + n_k U_\epsilon\|^{2_\alpha^*(s)}} \int_\Omega \frac{|u_1 + n_k U_\epsilon|^{2_\alpha^*(s)}}{|x|^s} dx\\
& = \frac{1}{\| \frac{u_1}{n_k} +  U_\epsilon\|^{2_\alpha^*(s)}} \int_\Omega \frac{| \frac{u_1}{n_k} +  U_\epsilon|^{2^*(s)_\alpha}}{|x|^s} dx.
\end{align*}

Hence,

\begin{align*}
\int_{\Omega}  \frac{|v_k|^{{2^*_\alpha(s)}}}{|x|^{s}} dx \to  \frac{1}{\|  U_\epsilon\|^{2_\alpha^*(s)}} \int_\Omega \frac{|  U_\epsilon|^{2^*(s)_\alpha}}{|x|^s} dx >0 \quad \text{ as } k \to \infty.
\end{align*}

On the other hand, $\text{ as } k \to \infty,$ we have 
 
\begin{align*}
&\I(t^- (v_k) v_k)\\
& \quad  = \frac{1}{2} [t^- (v_k)]^2 \| v_k\|^2 -\frac{\lambda}{q}  [t^- (v_k)]^q \int_{\Omega}  f(x) |v_k|^q dx - \frac{1}{2_{\alpha}^*(s)}[t^- (v_k)]^{2_{\alpha}^*(s)} \int_{\Omega}  \frac{|v_k|^{2_{\alpha}^*(s)}}{|x|^{s}} dx  \to -\infty. 
\end{align*}

This contradicts the fact that $\I$ is bounded below. Thus, (\ref{0< t-(u_1+ n_0 eta u_epsilon/ ||u_1+ n_0 eta u_epsilon ||) <C}) holds.

Now let  $n_0 := \frac{|C^2 - \|u_1\|^2 |^{\frac{1}{2}}}{\| U_\epsilon \|} +1.$ So,

\begin{align*}
\|u_1 + n_0 U_\epsilon\|^2 & = \|u_1\|^2 + n_0^2 \|U_\epsilon \|^2  + 2 n_0 \left( C_{n,\alpha}\langle  u_1, U_\epsilon \rangle_{\X} - \gamma \int_{\Omega} \frac{u_1 U_\epsilon}{|x|^\alpha} dx  \right)\\
& \ge \|u_1\|^2+ \left| C^2 - \|u_1\|^2\right| \\
& \ge C^2 \\
& > \left|   t^- (\frac{u_1 + n_0 U_\epsilon}{\|u_1 + n_0 U_\epsilon\|})  \right|^2,
\end{align*}

which gives

$$ t^- (\frac{u_1 + n_0 U_\epsilon}{\|u_1 + n_0 U_\epsilon\|}) <  \|u_1 + n_0 U_\epsilon\|.$$

This proves that   $u_1 + n_0 U_\epsilon \in W_2.$

Now define 

$$\Gamma:= \left\{ \tau \in C([0,1], \X) : \tau(0) = u_1 \text{ and } \tau(1)= u_1 + n_0 U_\epsilon  \right\},$$

$$ c^\star:= \inf\limits_{\tau \in \Gamma} \max\limits_{\xi \in [0,1]} \I(\tau(\xi)) \ \text{ and } \
{\gamma^\star}(\xi) = u_1 + \xi  n_0 U_\epsilon \quad \text{ for } \xi \in [0,1].$$

We have $\gamma^\star(0) \in W_1$ and $\gamma^\star (1) \in W_2.$ So, there exists $\xi_0 \in (0,1)$ such that $\gamma^\star(\xi_0) \in \Ne^-$ and $c^\star  \ge \M^-.$ It also follows from Lemma \ref{Lemma - any sequence which satisfies two conditions has strongly convergence in X} that 

$$ \M^- < c^\star < \M + \frac{\alpha -s}{2(n-s)} S_p^{\frac{n-s}{\alpha-s}}.$$

The Ekeland's variational principle yields that there exists a sequence $(u_k)_{k \in \N} \subset \N^-$ such that 

$$ \I(u_k) = \M^- +o(1) \quad \text{ and } \ \I'(u_k) =o(1) \  \text{ in } (\X)'. $$

We use Lemma \ref{Lemma - any sequence which satisfies two conditions has strongly convergence in X} and  (\ref{M- < M + alpha-s/2(n-s) (s_p)^ n-s/alpha-s }) to get that there exist a sub-sequence $(u_k)_{k \in \N}$ and $u_2$ such that 
$u_k \to u_2$ strongly in $\X.$ So, we have  that $u_2 \in \Ne^-$ and $\I(u_k) \to \I(u_2)=\M^-  \text{ as } k \to \infty.$

Since $\I(u_2) = \I(|u_2|),$ and $|u_2| \in \Ne^-$ is a solution for (\ref{Main problem-critical}), without loss of generality, we may assume  that $u_2$ is a non-negative solution for (\ref{Main problem-critical}), and the maximum principle \cite[Proposition 2.2.8]{Silvestre} implies that $u_2>0$ in $\Omega.$

\end{proof}

\begin{proof}[ Proof of Theorem \ref{Theorem - Critical}]

It follows from Theorem \ref{Theorem - existence of the solution in N+ - critical} and Proposition \ref{Proposition - existence of the solution in N- - critical}
that there exist two positive solutions $u_1$ and $u_2$ such that $u_1 \in \Ne^+$  and $u_2 \in \Ne^-.$ In addition, we have $\Ne^+ \cap \Ne^- = \emptyset.$ Thus,  $u_1$ and $u_2$ are two  distinct positive solutions for (\ref{Main problem-critical}). 

\end{proof}

\end{document}